\documentclass{article} 
\usepackage{amssymb}
\usepackage{amsmath}
\usepackage{amsthm}
\usepackage[a4paper]{geometry}
\usepackage[round]{natbib}
\usepackage[colorlinks=true, linkcolor=blue]{hyperref}

\usepackage{graphicx} 

\usepackage{amssymb}
\usepackage{amsmath,amscd}
\usepackage{amscd}
\usepackage{amsthm}
\usepackage{cancel}     %
\usepackage{centernot}  %
\usepackage{graphicx}   %
\usepackage{color}      %
\usepackage{enumitem}   %
\usepackage{commath}    %
\usepackage{algorithm}  %
\usepackage{booktabs}   %

\DeclareFontFamily{U}{mathx}{\hyphenchar\font45}
\DeclareFontShape{U}{mathx}{m}{n}{
      <5> <6> <7> <8> <9> <10> gen * mathx
      <10.95> mathx10 <12> <14.4> <17.28> <20.74> <24.88> mathx12
      }{}
\DeclareSymbolFont{mathx}{U}{mathx}{m}{n}
\DeclareMathSymbol{\intop}  {1}{mathx}{"B3}

\newcommand\indep{\independent}
\newcommand\independent{\protect\mathpalette{\protect\independenT}{\perp}}
\def\independenT#1#2{\mathrel{\rlap{$#1#2$}\mkern4mu{#1#2}}}

\newcommand{\wh}{\widehat}

\let\temp\phi
\let\phi\varphi
\let\varphi\temp

\renewcommand{\sec}{\textsection}

\newcommand{\E}{\mathbb{E}}

\newcommand{\given}{\,|\,}  %

\newcommand{\les}{\lesssim}

\DeclareMathOperator*{\argmax}{arg\,max}

\DeclareMathOperator{\BernoulliDist}{Ber}

\DeclareMathOperator{\pa}{pa} %
\DeclareMathOperator{\an}{an} %
\DeclareMathOperator{\nd}{nd} %
\DeclareMathOperator{\de}{de} %

\newcommand{\AND}{\text{ and }}

\newtheorem{theorem}{Theorem}[section]
\newtheorem{lemma}[theorem]{Lemma}
\newtheorem{corollary}[theorem]{Corollary}
\newtheorem{proposition}[theorem]{Proposition}

\newtheorem*{theorem*}{Theorem}
\newtheorem*{lemma*}{Lemma}
\newtheorem*{corollary*}{Corollary}
\newtheorem*{proposition*}{Proposition}
\newtheorem*{conjecture*}{Conjecture}

\theoremstyle{definition}
\newtheorem{definition}{Definition}
\newtheorem*{definition*}{Definition}
\theoremstyle{definition}
\newtheorem{example}{Example}
\theoremstyle{definition}
\newtheorem*{example*}{Example}
\theoremstyle{definition}

\theoremstyle{definition}
\newtheorem*{assumption*}{Assumption}
\theoremstyle{definition}
\newtheorem{condition}{Condition}
\theoremstyle{remark}
\newtheorem{remark}{Remark}
\theoremstyle{remark}
\newtheorem*{remark*}{Remark}

\newcommand{\gr}{G}%
\newcommand{\subgr}[1]{\gr[#1]}
\newcommand{\rv}{X}

\newcommand{\truepr}{P}
\newcommand{\ent}{H}
\newcommand{\mi}{I}
\newcommand{\layer}{L}
\newcommand{\anc}{A}

\newcommand{\estgr}{\wh{\gr}}

\DeclareMathOperator{\MB}{MB}
\newcommand{\mkvbdy}{m}
\newcommand{\mbsize}{M}
\newcommand{\width}{w}
\newcommand{\depth}{r}
\newcommand{\entgap}{\Delta}
\newcommand{\ppsgap}{\xi}

\newcommand{\ppsthresh}{\kappa}
\newcommand{\unentgap}{\eta}
\newcommand{\unentthresh}{\omega}
\newcommand{\ord}{\prec}

\newcommand{\anprime}{\an_j^i(k)}
\DeclareMathOperator{\imp}{im}

\title{Efficient Bayesian network structure learning via local Markov boundary search}

\usepackage{authblk}
\author[]{Ming Gao}
\author[]{Bryon Aragam}
\affil[]{\emph{University of Chicago}}

\begin{document}

\maketitle

\begin{abstract}%
We analyze the complexity of learning directed acyclic graphical models from observational data in general settings without specific distributional assumptions. 
Our approach is information-theoretic and uses a local Markov boundary search procedure in order to recursively construct ancestral sets in the underlying graphical model. 
Perhaps surprisingly, we show that for certain graph ensembles, a simple forward greedy search algorithm (i.e. without a backward pruning phase) suffices to learn the Markov boundary of each node.
This substantially improves the sample complexity, which we show is at most polynomial in the number of nodes.
This is then applied to learn the entire graph under a novel identifiability condition that generalizes existing conditions from the literature.
As a matter of independent interest, we establish finite-sample guarantees for the problem of recovering Markov boundaries from data. Moreover, we apply our results to the special case of polytrees, for which the assumptions simplify, and provide explicit conditions under which polytrees are identifiable and learnable in polynomial time.
We further illustrate the performance of the algorithm, which is easy to implement, in a simulation study.
Our approach is general, works for discrete or continuous distributions without distributional assumptions, and as such sheds light on the minimal assumptions required to efficiently learn the structure of directed graphical models from data.
\end{abstract}

\section{Introduction}
\label{sec:intro}

Learning the structure of a distribution in the form of a graphical model is a classical problem in statistical machine learning, whose roots date back to early problems in structural equations and covariance selection \citep{wright1921,wright1934,dempster1972}. 
Graphical models such as Markov networks and Bayesian networks impose structure in the form of an undirected graph (UG, in the case of Markov networks) or a directed acyclic graph (DAG, in the case of Bayesian networks). 
This structure is useful for variety of tasks ranging from querying and sampling to inference of conditional independence and causal relationships, depending on the type of graph used. 
In practice, of course, this structure is rarely known and we must rely on \emph{structure learning} to first infer the graphical structure.
The most basic version of this problems asks, given $n$ samples from some distribution $\truepr$ that is represented by a graphical model $\gr=(V,E)$, whether or not it is possible to reconstruct $\gr$.

In this paper, we study the structure learning problem for Bayesian networks (BNs). 
Our main contribution is a fine-grained analysis of a polynomial time and sample complexity algorithm for learning the structure of BNs with potentially unbounded maximum in-degree and without faithfulness. 
In particular, in our analysis we attempt to expose the underlying probabilistic assumptions that are important for these algorithms to work, drawing connections with existing work on local search algorithms and the conditional independence properties of $\truepr$.

\subsection{Contributions}
\label{sec:intro:contrib}

One of the goals of the current work is to better understand the minimal assumptions needed to identify and learn the structure of a DAG from data. Although this is a well-studied problem, existing theoretical work (see Section~\ref{sec:intro:related}) relies on assumptions that, as we show, are not really necessary. In particular, our results emphasize generic probabilistic structure (conditional independence, Markov boundaries, positivity, etc.) as opposed to parametric or algebraic structure (linearity, additivity, etc.), and hence provide a more concrete understanding of the subtle necessary conditions for the success of this approach.

With this goal in mind, we study two fundamental aspects of the structure learning problem: Identifiability and Markov boundary search. On the one hand, we provide a weaker condition for identifiability compared to previous work, and on the other, we exhibit families of DAGs for which forward greedy search suffices to provably recover the parental sets of each node. More specifically, our contributions can be divided into several parts: 
\begin{enumerate}
\item \emph{Identifiability} (Theorem~\ref{thm:ident:unequal2:short}).
We prove a new identifiability result on DAG learning. Roughly speaking, this condition requires that the entropy conditioned on an ancestral set $H(X_{k}\given \anc)$ of each node in $\gr$ is dominated by one of its ancestors.
An appealing feature of this assumption is that it applies to general distributions without parametric or structural assumptions, and generalizes existing ones based on second moments to a condition on the local entropies in the model. We also discuss in depth various relaxations of this and other conditions (Appendix~\ref{app:ext}).
\item \emph{Local Markov boundary search} (Algorithm~\ref{alg:pps}, Proposition~\ref{prop:pps}).
We prove finite-sample guarantees for a Markov boundary learning algorithm that is closely related to the incremental association Markov blanket (IAMB) algorithm, proposed in \citet{tsamardinos2003algorithms}. 
These results also shed light on the assumptions needed to successfully learn Markov boundaries in general settings; in particular, we do not require faithfulness, which is often assumed. 
\item \emph{Structure learning} (Algorithm~\ref{alg:uneq:short}, Theorem~\ref{thm:main}).
We propose an algorithm which runs in $O(d^3\depth\log d)$ time and $O(d^2\depth\log^{3} d)$ sample complexity, to learn an identifiable DAG $\gr$ from samples. Here, $d$ is the dimension and $\depth\le d$ is the depth of the DAG $\gr$, defined in Section~\ref{sec:bg}. 
\item \emph{Learning polytrees} (Theorem~\ref{thm:poly}). As an additional application of independent interest, we apply our results to the problem of learning polytrees \citep{dasgupta1999polytree,rebane2013recovery}.
\item \emph{Generalizations and extensions} (Appendix~\ref{app:ext}). We have included an extensive discussion of our assumptions with many examples and generalizations to illustrate the main ideas. For example, this appendix includes relaxations of the positivity assumption on $\truepr$, the main identifiability condition (Condition~\ref{cond:ident:main}), the PPS condition (Condition~\ref{cond:pps}), and extensions to general, non-binary distributions. We also discuss examples of the conditions and a comparison to the commonly assumed faithfulness condition.
\end{enumerate}
Finally, despite a long history of related work on Markov blanket learning algorithms \citep[e.g.][]{aliferis2010a,pena2007towards,statnikov2013algorithms}, to the best of our knowledge there has been limited theoretical work on the finite-sample properties of IAMB and related algorithms. It is our hope that the present work will serve to partially fill in this gap.

\subsection{Related work}
\label{sec:intro:related}

Early approaches to structure learning assumed faithfulness (for the definition, see Appendix~\ref{app:gm}; this concept is not needed in the sequel), which allows one to learn the Markov equivalence class of $\truepr$ \citep{spirtes1991,heckerman1995,lam1994,friedman1999,chickering2003}. 
Under the same assumption and assuming additional access to a black-box query oracle, \citet{barik2019learning} develop an algorithm for learning discrete BNs. 
\citet{barik2020provable} develop an algorithm for learning the undirected skeleton of $\gr$ without assuming faithfulness.
On the theoretical side, the asymptotic sample complexity of learning a faithful BN has also been studied \citep{friedman1996,zuk2012}. \citet{brenner2013} propose the SparsityBoost score and prove a polynomial sample complexity result, although the associated algorithm relies on solving a difficult integer linear program. 
\citet{chickering2002} study structure learning without faithfulness, although this paper does not establish finite-sample guarantees.
\citet{zheng2018dags,zheng2020learning} transform the score-based DAG learning problem into continuous optimization, but do not provide any guarantees.
\citet{aragam2019globally} analyze the sample complexity of score-based estimation in Gaussian models, although this estimator is based on a nonconvex optimization problem that is difficult to solve.

An alternative line of work, more in spirit with the current work, shows that $\gr$ itself can be identified without assuming faithfulness \citep{shimizu2006,hoyer2009,zhang2009,peters2013}, although these methods lack finite-sample guarantees. Recently, \citet{ghoshal2017ident} translated the equal variance property of \citet{peters2013} into a polynomial-time algorithm for linear Gaussian models with polynomial sample complexity. Around the same time, \citet{park2017} developed an efficient algorithm for learning a special family of distributions with 
\emph{quadratic variance functions}. To the best of our knowledge, these algorithms were the first provably polynomial-time algorithm for learning DAGs that did not assume faithfulness.
See also \citet{ordyniak2013} for an excellent discussion of the complexity of BN learning. The algorithm of \citet{ghoshal2017ident} has since been extended in many ways \citep{ghoshal2017sem,chen2018causal,wang2018nongauss,gao2020npvar}. 

Our approach is quite distinct from these approaches, although as we show, our identifiability result subsumes and generalizes existing work on equal variances.
Additionally, we replace second-moment assumptions with entropic assumptions, which are weaker and have convenient interpretations in terms of conditional independence, which is natural in the setting of graphical models. 
Since many of the analytical tools for analyzing moments and linear models are lost in the transition to discrete models, our work relies on fundamentally different (namely, information-theoretic) tools in the analysis.
Our approach also has the advantage of highlighting the important role of several fundamental assumptions that are somewhat obscured by the linear case, which makes strong use of the covariance structure induced by the linear model.
Finally, we note that although information-theoretic ideas have long been used to study graphical models \citep[e.g.][]{janzing2010causal,janzing2012information,kocaoglu2017entropic,kocaoglu2017greedy}, these works do not propose efficient algorithms with finite-sample guarantees, which is the main focus of our work.

\section{Preliminaries}
\label{sec:bg}

\paragraph{Notation} We use $[d]=\{1,\ldots,d\}$ to denote an index set.
As is standard in the literature on graphical models, we identify the vertex set of a graph $\gr=(V,E)$ with a random vector $\rv=(\rv_{1},\ldots,\rv_{d})$, and in the sequel we will frequently abuse notation by identifying $V=X=[d]$. 
For any subset $S\subset V$, $\subgr{S}$ is the subgraph defined by $S$. 
Given a DAG $\gr=(V,E)$ and a node $\rv_{k}\in V$, $\pa(k)=\{\rv_{j}:(j,k)\in E\}$ is the set of parents, $\de(k)$ is the set of descendants, $\nd(k):=V\setminus\de(k)$ is the set of nondescendants, and $\an(k)$ is the set of ancestors. 
Analogous notions are defined for subsets of nodes in the obvious way. 
A source node is any node $\rv_{k}$ such that $\an(k)=\emptyset$ and a sink node is any node $\rv_{k}$ such that $\de(k)=\emptyset$. 
Every DAG admits a unique decomposition into $r$ layers, defined recursively as follows: $\layer_{1}$ is the set of all source nodes of $\gr$ and $\layer_{j}$ is the set of all sources nodes of $\subgr{V\setminus\cup_{t=0}^{j-1}\layer_{t}}$. 
By convention, we let $\layer_{0}=\emptyset$ and layer width $d_j=|\layer_j|$, the largest width $\max_jd_j=\width$. 
An ancestral set is any subset $\anc\subset V$ such that $\an(\anc)\subset \anc$. 
The layers of $\gr$ define canonical ancestral sets by $\anc_{j}=\cup_{t=0}^{j}\layer_{t}$. 
Finally, the Markov boundary of a node $X_{k}$ relative to a subset $S\subset V$ is the smallest subset $m\subset S$ such that $X_{k}\indep (S\setminus m) \given m$, and is denoted by $\MB(\rv_{k};S)$ or $\mkvbdy_{Sk}$ for short.

The entropy of a discrete random variable $Z$ is given by $\ent(Z)=-\sum_{z}\truepr(Z=z)\log \truepr(Z=z)$, the conditional entropy given $Y$ is $\ent(Z\given Y)=-\sum_{z,y}\truepr(Z=z, Y=y)\log \truepr(Z=z\given Y=y)$, and the mutual information between $Z$ and $Y$ is $\mi(Z;Y)=\ent(Z) - \ent(Z\given Y)$. For more background on information theory, see \citet{cover2012elements}.

\paragraph{Graphical models}

Let $\rv=(\rv_{1},\ldots,\rv_{d})$ be a random vector with distribution $\truepr$. In the sequel, for simplicity, we assume that $\rv\in\{0,1\}^{d}$ and that $\truepr$ is strictly positive, i.e. $\truepr(X=x)>0$ for all $x\in\{0,1\}^{d}$. These assumptions are not necessary; see Appendix~\ref{app:ext} for extensions to categorical random variables and/or continuous random variables and nonpositive distributions.

A DAG $\gr=(V,E)$ is a \emph{Bayesian network} (BN) for $\truepr$ if $\truepr$ factorizes according to $\gr$, i.e.
\begin{align}
\label{eq:defn:bn}
\truepr(\rv)
= \prod_{k=1}^{d} \truepr(\rv_{k}\given\pa(k)).
\end{align}
Obviously, by the chain rule of probability, a BN is not necessarily unique---any permutation of the variables can be used to construct a valid BN according to \eqref{eq:defn:bn}. A \emph{minimal I-map} of $\truepr$ is any BN such that removing any edge would violate \eqref{eq:defn:bn}. The importance of the factorization \eqref{eq:defn:bn} is that it implies that separation in $\gr$ implies conditional independence in $\truepr$, and a minimal I-map encodes as many such independences as possible (although not necessarily all independences). 
For more details, a review of relevant graphical modeling concepts is included, see Appendix~\ref{app:gm}.

The purpose of structure learning is twofold: 1) To identify a unique BN $\gr$ which can be identified from $\truepr$, and 2) To devise algorithms to learn $\gr$ from data. This is our main goal.

\section{Identifiability of minimal I-maps}
\label{sec:ident}

In this section we introduce our main assumption regarding identifiability of minimal I-maps of $\truepr$.

\subsection{Conditions for identifiability}

For $X_k \in V\setminus \anc_{j+1}$, define $\an_j(k):=\an(X_{k})\cap L_{j+1}$, i.e $\an_j(k)$ denotes the ancestors of $X_k$ in $\layer_{j+1}$. By convention, we let $\anc_{0}=\emptyset$. Finally, with some abuse of notation, let $\layer(X_{k})\in[r]$ indicate which layer $X_{k}$ is in, i.e. $\layer(X_{k})=j$ if and only if $X_{k}\in\layer_{j}$.
\begin{condition}
\label{cond:ident:main}

For each $X_{k}\in V$ with $\layer(X_k)\ge 2$ and for each $j=0,\cdots,\layer(X_{k})-2$, there exists $X_i \in \an_{j}(k)$ such that following two conditions hold:
\begin{enumerate}[label=(C\arabic*), wide = 0pt]
\item $\ent(X_i\given \anc_{j}) < \ent(X_k \given \anc_{j})$, \label{cond:uneq:short}
\item $I(X_k;X_i \given \anc_{j}) > 0$. \label{cond:uneq:mi:short}
\end{enumerate}
\end{condition}
We refer $X_i\in\an_j(k)$ satisfying \ref{cond:uneq:short} and \ref{cond:uneq:mi:short} in Condition~\ref{cond:ident:main} as the \textit{important ancestors} for $X_k$, denoted by $\imp_{j}(k)$. Thus, another way of stating Condition~\ref{cond:ident:main} is that 
for each $X_{k}\in V$ with $\layer(X_k)\ge 2$ and for each $j=0,\cdots,\layer(X_{k})-2$, there exists an important ancestor, i.e. $\imp_{j}(k)\ne\emptyset$.

The idea behind this condition is the following: Suppose we wish to identify the $t$th layer of $\gr$. Condition~\ref{cond:ident:main} requires that for every node \emph{after} $\layer_{t}$ (represented by $X_{k}$),
there is at least one node $X_{i}$ in $\layer_{t}$ satisfying \ref{cond:uneq:short} and \ref{cond:uneq:mi:short}. 
The intuition is the entropy (uncertainty) in the entire system is required to increase as time evolves. 

Condition \ref{cond:uneq:short} is the operative condition: By contrast, \ref{cond:uneq:mi:short} is just a nondegeneracy condition that says $X_i\not \indep X_k\given \anc_j$, which will be violated only when $\an_j(k)$ is conditionally independent of $X_k$ given $\anc_{j}$. 
At the population level, \ref{cond:uneq:mi:short} is superficially similar to faithfulness, however, a closer look reveals significant differences.
One example of conditional independence between an entire layer of ancestors and descendants is path-cancellation, where effects through \emph{multiple} paths are neutralized through delicate choices of parameters,
whereas unfaithfulness occurs when there happens to be just one such path cancellation. Moreover, \ref{cond:uneq:mi:short} only applies to a small set of ancestral sets, whereas faithfulness applies to all possible $d$-separating sets.
Not only is this kind of path-cancellation for all of $\an_j(k)$ unlikely, we show in Appendix~\ref{app:lem:poly} that this is essentially the only way \ref{cond:uneq:mi:short} can be violated: If $\gr$ is a poly-forest, then \ref{cond:uneq:mi:short} always holds, see Lemma~\ref{lem:poly}.

In Section~\ref{sec:alg:int}, we will discuss this condition in the context of an algorithm and an example, which should help explain its purpose better. 
Before we interpret this condition further, however,
let us point out why this condition is important: It identifies $\gr$.

\begin{theorem}\label{thm:ident:unequal2:short}
If there is a minimal I-map $\gr$ satisfying Condition~\ref{cond:ident:main}, then $\gr$ is identifiable from $\truepr$. 
\end{theorem}
In order to better understand Condition~\ref{cond:ident:main}, let us first compare it to existing assumptions such as equal variances \citep{ghoshal2017ident,peters2013}. 
Indeed, there is a natural generalization of the equal variance assumption to Shannon entropy:
\begin{enumerate}[label=(C\arabic*), wide = 0pt]
\setcounter{enumi}{2}
\item $\ent(\rv_{k}\given \pa(k))\equiv h^*$ is the same for each node $k=1,\ldots,d$. \label{cond:equal}
\end{enumerate}

One reason to consider entropy is due to the fact that every distribution with $\E X^c < \infty$ for some $c>0$ has well-defined entropy, whereas not all distributions have finite variance. Though quoted here,
we will not require \ref{cond:equal} in the sequel; it is included here merely for comparison.
Indeed, the next result shows that this ``equal entropy'' condition is a special case of Condition~\ref{cond:ident:main}:
\begin{lemma}
\label{lem:equal2main}
Assuming \ref{cond:uneq:mi:short}, \ref{cond:equal} implies \ref{cond:uneq:short}. Thus, the ``equal entropy'' condition implies Condition~\ref{cond:ident:main}.
\end{lemma}

In fact, \ref{cond:uneq:short} significantly relaxes \ref{cond:equal}. The latter implies that all nodes in $\layer_{j+1}$ have smaller conditional entropy than $X_k$, whereas \ref{cond:uneq:short} only requires this inequality to hold for at least one ancestor $X_i\in\an_j(k)$. Moreover, something even stronger is true: The equal variance condition can be relaxed to \emph{unequal} variances (see Assumption~1 in \cite{ghoshal2017sem}), and we can derive a corresponding ``unequal entropy'' condition. This condition is also a special case of Condition~\ref{cond:ident:main}. We can also construct explicit examples that satisfy Condition~\ref{cond:ident:main}, but neither the equal nor unequal entropy condition. For details, see Appendix~\ref{app:ext:compare}. 

\begin{remark}
Condition~\ref{cond:ident:main} can be relaxed even further: See Appendix~\ref{app:ext:general} and Remark~\ref{rem:gen} for a discussion along with its corresponding algorithm.
\end{remark}

\subsection{Algorithmic interpretation}
\label{sec:alg:int}

The proof of Theorem~\ref{thm:ident:unequal2:short} motivates a natural algorithm to learn the DAG, shown in Algorithm~\ref{alg:uneq:short}. This algorithm exploits the fact that given $\anc_j$, nodes within $\layer_{j+1}$ are mutually independent. Algorithm~\ref{alg:uneq:short} is a layer-wise DAG learning algorithm. For each layer, it firstly sorts the conditional entropies in ascending order $\tau$, then runs a ``Testing and Masking'' (TAM) step to distinguish nodes in $\layer_{j+1}$ from remaining ones ($X_k$): We use $\imp_j(k)$ defined in \ref{cond:uneq:short} to detect and \textit{mask} $X_k \notin \layer_{j+1}$ by \textit{testing} conditional independence. 
By masking, we mean we do not consider the nodes being masked when proceeding over the entropy ordering $\tau$ to identify $\layer_{j+1}$.

\begin{algorithm}[t]
\caption{Learning DAG structure}
\label{alg:uneq:short}
\textbf{Input:} $X=(X_1,\ldots,X_d)$, $\unentthresh$\\
\textbf{Output:} $\widehat{\gr}$.
\begin{enumerate}
\item Initialize empty graph $\widehat{\gr}=\emptyset$ and $j=0$
\item Set $\wh{\layer}_j=\emptyset$, let $\wh{\anc}_{j}=\cup_{t=0}^j \wh{\layer}_t$
\item While $V\setminus \wh{\anc}_j\ne\emptyset$: %
\begin{enumerate}
\item For $k\notin \wh{\anc}_{j}$, estimate conditional entropy $\ent(X_k\given \wh{\anc}_{j})$ by some estimator $\wh{h}_{jk}$. 
\item Initialize $\wh{\layer}_{j+1} = \emptyset$, $\wh{S}_{j+1} = \emptyset$. Sort $\wh{h}_{jk}$ in ascending order and let $\wh{\tau}^{(0)}$ be the corresponding permutation of $V\setminus \wh{\anc}_j$. 
\item For $\ell\in 0,1,2,\ldots$ until $|\wh{\tau}^{(\ell)}|=0$: \textbf{[TAM step]}
\begin{enumerate}
\item  $\wh{\layer}_{j+1}=\wh{\layer}_{j+1}\cup \{\wh{\tau}^{(\ell)}_1\}$.
\item For $k\notin \wh{\anc}_{j}\cup \wh{\layer}_{j+1} \cup \wh{S}_{j+1}$, estimate $\mi(X_k; \wh{\tau}^{(\ell)}_1 \given \wh{\anc}_{j})$. %
\item Set $\wh{S}_{j+1} = \wh{S}_{j+1}\cup \{k: \wh{I}^{(\ell)}_{jk} \ge \unentthresh\}$.
\item $\wh{\tau}^{(\ell+1)}=\wh{\tau}^{(\ell)}\setminus \left(\wh{\layer}_{j+1} \cup \wh{S}_{j+1}\right)$
\end{enumerate}
\item\label{alg:uneq:short:pa} For $k\in \wh{\layer}_{j+1}$, set $\pa_{\wh{\gr}}(k)=\MB(X_{k};\wh{\anc}_{j})$. %
\item Update $j = j+1$.
\end{enumerate}
\item Return $\widehat{\gr}$.
\end{enumerate}
\end{algorithm}

\begin{figure}[t]
    \centering
    \includegraphics[width=1.\linewidth]{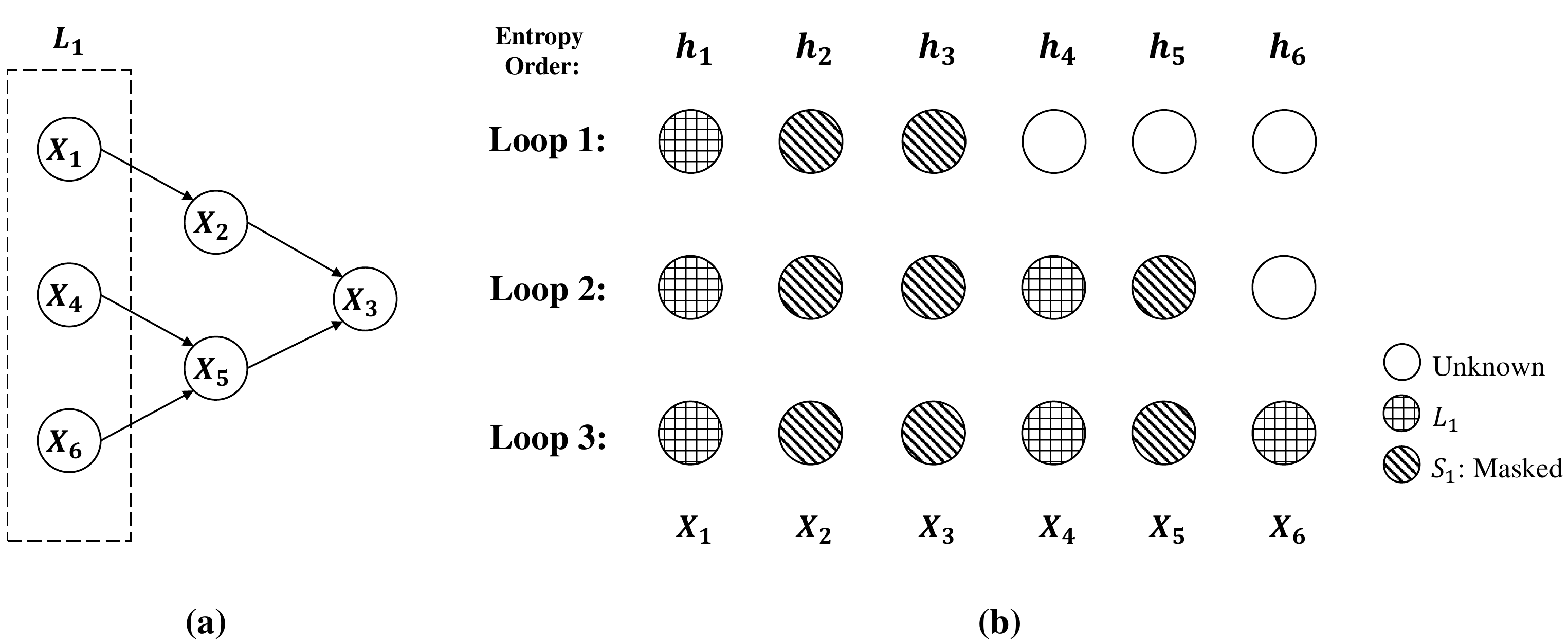}
    \caption{Example for TAM algorithm: (a) True DAG with $\layer_1$ outlined; (b) Status of the algorithm after each loop, denoted by different patterns of nodes.}
    \label{fig:uneq:example}
\end{figure}

In order to see how Algorithm~\ref{alg:uneq:short} works, consider the example shown in Figure~\ref{fig:uneq:example}(a). In the first step with $j=0$, we use marginal entropy (i.e. since $\anc_{0}=\emptyset$) to distinguish $\layer_{1}$ from the remaining nodes. 
Let $H(X_{\ell}):=h_{\ell}$ and assume for simplicity that the nodes are ordered such that $h_1 < h_2 < \cdots < h_6$ (Step 3(b)). 
Apparently, the inequalities that $h_3 < h_{[4:6]}$ and $h_5<h_6$ imply \ref{cond:equal} does not hold here. 
Suppose there are no spurious edges, i.e. descendants and ancestors are dependent. Now we can see conditions in Theorem~\ref{thm:ident:unequal2:short} are satisfied and the important ancestors for $X_2,X_3,X_5$ are $X_1,X_1,X_4$ respectively. The implementation of Algorithm~\ref{alg:uneq:short} is visualized in Figure~\ref{fig:uneq:example}(b): In the first loop, $X_1$ is taken into $\layer_1$ and $X_2,X_3$ are masked due to dependence (Step 3c(iii)). In the second loop, $X_4$ is added to $\layer_{1}$ and then $X_5$ is masked. Finally, with $X_6$ put into $\layer_1$, we have identified $\layer_1$. 

It is worth emphasizing that the increasing order of marginal entropies in this example is purely for simplicity of presentation, and does not imply any information on the causal order of the true DAG. The marginal entropies of nodes need not be monotonic with respect to the topological order of $\gr$.

\section{Local Markov boundary search}
\label{sec:mb}

Algorithm~\ref{alg:uneq:short} assumes that we can learn $\MB(X_{k};\anc_{j})$, the Markov boundary of $X_{k}$ in the ancestral set $\anc_{j}$. This is a well-studied problem in the literature, and a variety of greedy algorithms have been proposed for learning Markov boundaries from data \citep{tsamardinos2003algorithms,tsamardinos2003time,aliferis2010a,pena2007towards,fu2008fast,gao2016efficient}, all based on the same basic idea: Greedily add variables whose association with $X_{k}$ is the highest. In this section, we establish theoretical guarantees for such a greedy algorithm.
In the next section, we apply this algorithm to reconstruct to full DAG $\gr$ via Algorithm~\ref{alg:uneq:short}.

To present our Markov boundary search algorithm, we first need to set the stage. Let $X_{k}\in V$ be any node and $\anc$ an ancestral set of $X_{k}$.
We wish to compute $\MB(X_{k};\anc)$ and $\ent(X_{k}\given\anc)$. An algorithm for this is outlined in Algorithm~\ref{alg:pps}.
In contrast to many existing local search methods for learning Markov boundaries, this algorithm is guaranteed to return the parents of $X_{k}$ in $\gr$. 
More specifically, if $\rv_{k}\in\layer_{j+1}$, then $\MB(X_k;\anc_j)=\pa(\rv_{k})$. For this reason, we refer to Algorithm \ref{alg:pps} as \emph{possible parent selection}, or PPS for short. In fact,  PPS is exactly the forward phase of the well-known IAMB algorithm for Markov blanket discovery \citep{tsamardinos2003algorithms} with conditional mutual information used both as an association measure and as a conditional independence test.

\begin{algorithm}[t]
\caption{Possible Parental Set (PPS) procedure}
\label{alg:pps}
\textbf{Input:} $X=(X_1,\ldots,X_d)$, $k$, $\anc$, $\ppsthresh$ \\
\textbf{Output:} Conditional entropy $\wh{h}$, Markov boundary $\wh{\mkvbdy}$
\begin{enumerate}
\item Initialize $\wh{\mkvbdy}=\emptyset$, loop until $\wh{\mkvbdy}$ does not change:
\begin{enumerate}
    \item\label{alg:pps:mi} For $\ell\in \anc\setminus \wh{\mkvbdy}$, estimate conditional mutual information $I(X_\ell;X_k\given \wh{\mkvbdy})$ by some estimator $\wh{I}_{\ell}$.
    \item Let $\ell^*=\argmax_{\ell\notin \anc\setminus \wh{\mkvbdy}}\wh{I}_{\ell}$, if $\wh{I}_{\ell^*} > \ppsthresh$, set $\wh{\mkvbdy}=\wh{\mkvbdy}\cup \{\ell^*\}$.
\end{enumerate}
\item\label{alg:pps:condent} Estimate conditional entropy $\ent(X_k\given \wh{\mkvbdy})$ by some estimator $\wh{h}$.
\item Return conditional entropy estimation $\wh{h}$ and estimated Markov boundary $\wh{\mkvbdy}$.
\end{enumerate}
\end{algorithm}

Although PPS will always return a valid Markov \emph{blanket}, without a backward phase to remove unnecessary variables added by the forward greedy step, PPS may fail to return a \emph{minimal} Markov blanket, i.e. the Markov boundary. The following condition is enough to ensure no unnecessary variables are included:
\begin{condition}[PPS condition]
\label{cond:pps}
For any proper subset $m\subsetneq \MB(X_k;\anc)$ 
and any node $X_{\ell}\in \anc\setminus \MB(\rv_{k};\anc)$, there exists $X_{c}\in \MB(\rv_{k};\anc)\setminus m$ such that
\[
I(X_k;X_{c}\given m) > I(X_k;X_{\ell}\given m).
\]
\end{condition}
\noindent 
Condition~\ref{cond:pps} requires that nodes in $\MB(\rv_{k};\anc)$ always contribute larger conditional mutual information marginally than those that are not in $\MB(\rv_{k};\anc)$. Thus when we do greedy search to select parents in Algorithm~\ref{alg:pps}, only the nodes in $\MB(\rv_{k};\anc)$ will be chosen. Therefore, with a proper threshold $\ppsthresh$, this set can be perfectly recovered without incorporating any nuisance nodes.

We now present the sample complexity of using PPS to recover Markov boundaries under Condition~\ref{cond:pps}.
We will make frequent use of the Markov boundary and its size:
\begin{align}
\label{eq:defb:mkvbdy}
\mkvbdy_{\anc k}
:= \MB(\rv_{k};\anc), \ \ \ \ \mbsize_{\anc k}:=|\mkvbdy_{\anc k}|.
\end{align}
In particular, by definition we have $\rv_{k}\indep (\anc\setminus\mkvbdy_{\anc k})\given \mkvbdy_{\anc k}$. Under Condition~\ref{cond:pps}, we further define following quantities:
\begin{align}\label{eq:defn:ppsgap}
\widetilde{\entgap}_{\anc k} & := \min_{m\subsetneq \mkvbdy_{\anc k}}\left[\max_{c\in \mkvbdy_{\anc k}\setminus m} I(X_k;X_c\given m) - \max_{\ell\in \anc_{j}\setminus \mkvbdy_{\anc k}} I(X_k;X_{\ell}\given m)\right] > 0, \nonumber \\
 \ppsgap_{\anc k} & := \min_{m\subsetneq \mkvbdy_{\anc k}}\min_{c\in \mkvbdy_{\anc k}\setminus m}I(X_k;X_c\given m)/2 > 0.
\end{align}
$\widetilde{\entgap}_{\anc k}$ is the gap between the mutual information of nodes inside or outside of $\mkvbdy_{\anc k}$. $\ppsgap_{\anc k}$ is the minimum mutual information contributed by the nodes in $\mkvbdy_{\anc k}$. The larger these quantities, the easier the $\MB(X_k;\anc)$ is to be recovered. 
\begin{proposition}
\label{prop:pps}
Fix $k\in V$ and let $\anc$ be any set of ancestors of $X_{k}$.
Suppose Condition~\ref{cond:pps} holds. Algorithm~\ref{alg:pps} is applied for to estimate $\mkvbdy_{\anc k}$ and $\ent(X_k\given \anc)$ with $\ppsthresh \le \ppsgap_{\anc k}$, we have with $t \le \min\left(\ppsthresh, \tilde{\entgap}_{\anc k}/2\right)$
\[
\truepr\left(\wh{\mkvbdy}_{\anc k}=\mkvbdy_{\anc k}, \big|\widehat{\ent}(X_k\given \anc)-\ent(X_k\given \anc)\big|<t\right)\ge 1-(\mbsize_{\anc k}+2)|\anc| \frac{\delta^2_{\mbsize_{\anc k}}}{t^2}.
\]
where $\delta^2_{\mbsize_{\anc k}}$ is the estimation error of conditional entropy defined by \eqref{eq:def:delta} in Appendix~\ref{app:main:prebound}, which depends on $\mbsize_{\anc k}$.
\end{proposition}
A na\"ive analysis of this algorithm would have a sample complexity that depends on the size of the ancestral set $\anc$; note that our more fine-grained analysis depends instead on the size of the Markov boundary $\MB(X_{k};\anc)$.
We assume with sample size large enough, $\delta^2_{\mbsize_{\anc k}}$ is small such that the right hand side remains to be positive and goes to $1$. 
The proof of Proposition~\ref{prop:pps} is deferred to Appendix~\ref{app:samcom:pps}. 

Condition~\ref{cond:pps} ensures the success of the greedy PPS algorithm. Although this assumption is not strictly necessary for structure learning (see Appendix~\ref{app:ext:pps} for details), it significantly improves the sample complexity of the structure learning algorithm (Algorithm~\ref{alg:uneq:short}). Thus, it is worthwhile to ask when Condition~\ref{cond:pps} holds: We will take up this question again in Section~\ref{sec:dag:pps}.

\section{Learning DAGs}
\label{sec:dag}

Thus far, we have accomplished two important subtasks for learning a DAG: In Section~\ref{sec:ident}, we identified its layer decomposition $\layer=(\layer_{1},\ldots,\layer_{r})$. In Section~\ref{sec:mb}, we showed that the PPS procedure successfully recovers (ancestral) Markov boundaries. Combining these steps, we obtain a complete algorithm for learning $\gr$.
In this section, we study the computational and sample complexity of this algorithm; proofs are deferred to the appendices.

We adopt the notations for Markov boundaries as in \eqref{eq:defb:mkvbdy}:
\begin{align}
\label{eq:defb:mkvbdy:layer}
\mkvbdy_{jk}
:= \MB(\rv_{k};\anc_{j}) \ \ \ \ \mbsize_{jk}:=|\mkvbdy_{jk}|
\end{align}
Therefore, $\rv_{k}\indep (\anc_{j}\setminus\mkvbdy_{jk})\given \mkvbdy_{jk}$.
A critical quantity in the sequel will be the size of the largest Markov boundary $\mkvbdy_{jk}$, which we denote by $\mbsize$: 
\begin{align}
\label{eq:defn:mbsize:largest}
\mbsize
:=\max_{jk}\mbsize_{jk} = \max_{jk}|\mkvbdy_{jk}|.
\end{align}
This quantity depends on the number of nodes $d$ and the structure of the DAG. For example, if the maximum in-degree of $\gr$ is 1, then $\mbsize=1$. A related quantity that appears in existing work is the size of the largest Markov boundary \emph{relative to all of $\rv$}, which may be substantially larger than $\mbsize$.
The former quantity includes both ancestors \emph{and} descendants, whereas $\mkvbdy_{jk}$ only contains ancestors. Analogously, let $\mbsize_j = \max_k\mbsize_{jk}$.

\subsection{Algorithm}
\label{sec:dag:alg}
By combining Algorithms~\ref{alg:uneq:short} and \ref{alg:pps}, we obtain a complete algorithm for learning the DAG $\gr$, which we refer to as the TAM algorithm, short for \emph{Testing and Masking}.
It consists of two parts: 
\begin{enumerate}
\item Learning the layer decomposition $(\layer_{1},\ldots,\layer_{r})$ by minimizing conditional entropy and TAM step (Algorithm \ref{alg:uneq:short});
\item Learning the parental sets and reducing the size of conditioning sets by learning the Markov boundary 
$\mkvbdy_{jk} = \MB(\rv_{k};\anc_{j})$ (Algorithm \ref{alg:pps}).
\end{enumerate}

Specifically, we use PPS (Algorithm~\ref{alg:pps}) to estimate the conditional entropies (Step 3(a)), conditional mutual information (Step 3c(ii)), and the Markov boundary (Step 3(d)).
For completeness, the complete procedure is detailed in Algorithm \ref{alg:uneq:withpps} in Appendix~\ref{app:compalg}.
More generally, Algorithm~\ref{alg:pps} can be replaced with any Markov boundary recovery algorithm or conditional entropy estimator; this highlights the utility of Algorithm~\ref{alg:uneq:short} as a separate meta-algorithm for DAG learning.

One missing piece is the choice of estimators for conditional entropy and mutual information in steps~\eqref{alg:pps:mi} and~\eqref{alg:pps:condent} of the PPS procedure. We adopt the minimax entropy estimator from \citet{wu2016minimax} by treating (without loss of generality) the joint entropy as the entropy of a multivariate discrete variable, although other estimators can be used without changing the analysis. 
The complexity of this estimator is exponential in $\mbsize$ (i.e. since there are up to $2^{\mbsize}$ states to sum over in any Markov boundary $\mkvbdy_{jk}$), so the computational complexity of Algorithm~\ref{alg:pps} is $O(Md2^{M})$. In addition, for the TAM step, there are at most $\max_jd_j$ nodes in each layer to estimate conditional mutual information with remaining at most $d$ nodes, thus this step has computational complexity $O(d\max_j d_j 2^\mbsize)$. 
Thus assuming $\mbsize \lesssim \log d$, 
the overall computational complexity of Algorithm~\ref{alg:uneq:short} is at most $O\big(\depth \times (d\times \mbsize d2^\mbsize + d\times\max_j d_j 2^\mbsize)\big) = O(d^3\depth\log d)$. 
Specifically, for $\depth$ layers, we must estimate the conditional entropy of at most $d$ nodes, and call TAM step once.

\subsection{Main statistical guarantees}
\label{sec:dag:stat}
In order to analyze the sample complexity of Algorithm~\ref{alg:uneq:withpps} under Conditions~\ref{cond:ident:main} and ~\ref{cond:pps}, we introduce following positive quantities:
\begin{align*}
    \entgap&:= \min_j \min_{k\in V\setminus\anc_{j+1}}(\ent(X_k\given\anc_j) - \ent(X_i\given \anc_j) > 0 \\
    \unentgap&:= \min_j\min_{k\in V\setminus\anc_{j+1}}\mi(X_k;X_i\given \anc_j) > 0
\end{align*}    
where $X_i \in \imp_j(k)$ is defined in Condition~\ref{cond:ident:main}. These two quantities are corresponding to the two conditions \ref{cond:uneq:short} and \ref{cond:uneq:mi:short}, which are used to distinguish each layer with its descendants.
Compared to \textbf{strong} faithfulness, which is needed on finite samples, we only require a much smaller, restricted set of information measures to be bounded from zero.
We also adopt the quantities defined in \eqref{eq:defn:ppsgap} by setting $\anc=\anc_j$ and drop the notation $\anc$ such that $\widetilde{\entgap}_{jk}:=\widetilde{\entgap}_{\anc_j k}$ and $\ppsgap_{jk}:=\ppsgap_{\anc_j k}$.

Finally we are ready to state the main theorem about sample complexity of Algorithm~\ref{alg:uneq:short}:
\begin{theorem}
\label{thm:main}
Suppose $\truepr$ satisfies Conditions~\ref{cond:ident:main} and~\ref{cond:pps}, and let $\gr$ be the minimal I-map identified by Theorem~\ref{thm:ident:unequal2:short}.
Let $\estgr$ be output of Algorithm~\ref{alg:uneq:withpps} applied with
$\unentthresh\le \unentgap/2$ and  $\kappa\le\min_{jk}\ppsgap_{jk}$. Denote $\entgap^*_{\unentthresh,\ppsthresh}=\min_{jk}(\entgap/2,\unentthresh,\ppsthresh, \widetilde{\entgap}_{jk}/2)$.
If $\mbsize\lesssim \log d$ and 
\[
n\gtrsim \frac{d^2\depth\log^{3} d}{(\entgap^*_{\unentthresh,\ppsthresh})^2\epsilon},
\]
then $\estgr=\gr$ with probability $1-\epsilon$. 
\end{theorem}
\noindent 
Up to log factors, the sample complexity scales quadratically with the number of nodes $d^2$ and linearly in the depth $\depth\le d$. 
In the worst case, this is cubic in the dimension. 
For example, if $\gr$ is a Markov chain then $\mbsize=1$ and $\depth=d$, thus it suffices to have $n=\Omega(d^3\log^3 d)$.
Comparatively, most of previous work \citep{chen2018causal,ghoshal2018learning,wang2018nongauss} only consider linear or parametric models. 
One recent work that provides nonparametric guarantees without assuming faithfulness is \citep{gao2020npvar}, who show that in general, $\Omega((dr/\epsilon)^{1+d/2})$ samples suffice to recover the topological ordering under an equal variance assumption similar in spirit to \ref{cond:equal}. Unlike the current work, which considers exact recovery of the full graph $\gr$, \citep{gao2020npvar} does not include the reduction to Markov boundary search that is crucial to our exact recovery results.

In fact, as the proof indicates, the sample complexity of our result also depends exponentially on $\mbsize$.
This explains the assumption that $M\les\log d$; the logarithmic assumption is analogous to sparsity and results from not making any parametric assumptions on the model. Since our setting is fully nonparametric, exponential rates in the dimension $d$ are to be expected. 
Under stronger parametric assumptions, these exponential rates can likely be avoided.
A detailed analysis of the dependency on $\mbsize$ can be found in the proofs, which are deferred to Appendix \ref{app:main}.

Finally, in practice the quantities $\unentgap,\ppsgap_{jk}$ needed in Theorem~\ref{thm:main} may not be known, which hinders the choice of tuning parameters $\unentthresh,\ppsthresh$. In Appendix~\ref{app:tuning} (see Theorem~\ref{thm:tuning}) we discuss the selection of these tuning parameters in a data-dependent way.

\subsection{Application to learning polytrees}
\label{sec:dag:pps}

Learning polytrees is one of the simplest DAG learning problems. This problem was introduced in \citet{rebane2013recovery} and in a seminal work, \citet{dasgupta1999polytree} showed that learning polytrees is NP-hard in general. In this section, we note that when the underlying DAG is a polytree (or more generally, a polyforest), Condition~\ref{cond:pps} is always satisfied, and therefore we have an efficient algorithm for learning identifiable polyforests that satisfy Condition~\ref{cond:ident:main}.

Recall that a polyforest is a DAG whose skeleton (i.e. underlying undirected graph) has no cycles.
\begin{theorem}
\label{thm:poly}
If $\truepr$ satisfies \eqref{eq:defn:bn} for some polyforest $\gr$, then Condition~\ref{cond:pps} is satisfied. 
\end{theorem}
\noindent 
The detailed proof can be found in Appendix~\ref{app:poly}. As a result, it follows immediately (by combining Theorems~\ref{thm:main} and~\ref{thm:poly} with Lemma~\ref{lem:poly}) that any polytree satisfying \ref{cond:uneq:short} is learnable.

The crucial property of a poly-forest used in proving Theorem~\ref{thm:poly} is that there exists at most one directed path between any two nodes in the graph. 
However, the existence of multiple directed paths between two nodes does not necessarily imply that Condition~\ref{cond:pps} will fail: There are concrete examples of graphs satisfying Condition~\ref{cond:pps} with arbitrarily many paths between two nodes. An example is given by the DAG $\gr=(V,E)$ with $V=(Z,X_{1},\ldots,X_{k},Y)$ such that $Z\to X_{i}\to Y$ for each $i$. 
Thus, this assumption holds more broadly than suggested by Theorem~\ref{thm:poly}.
It is an interesting problem for future work to study this further.

\section{Experiments}
\label{sec:exp}
We conduct a brief simulation study to demonstrate the performance of Algorithm~\ref{alg:uneq:short} and compare against some common baselines: PC \citep{spirtes1991}, GES \citep{chickering2003}. We focus on the fully discrete setting. All implementation details can be found in Appendix~\ref{app:exp}. The code implementing TAM algorithm is available at \url{https://github.com/MingGao97/TAM}. We stress that the purpose of these experiments is simply to illustrate that the proposed algorithm can be implemented in practice, and successfully recovers the edges in $\gr$ as predicted by our theoretical results.

We evaluate the performance of aforementioned algorithms by Structural Hamming distance (SHD): This is a standard metric for DAG learning that counts the total number of edge additions, deletions, and reversals needed to convert the estimated graph into the reference one. Since PC and GES both return a CPDAG that may contain undirected edges, we evaluate these algorithms favourably by assuming correct orientation for undirected edges wherever they are present.

We simulate DAGs from three graph types: Poly-trees (Tree) , Erd\"os-R\'enyi (ER), and Scale-Free (SF) graphs. As discussed in Section~\ref{sec:dag:pps}, poly-tree models are guaranteed to satisfy Condition~\ref{cond:pps}, whereas general DAGs such as ER or SF graphs are not, so this provides a test case for when this condition may fail. We generate data according to two models satisfying the ``equal entropy'' condition \ref{cond:equal}. As discussed in Appendix~\ref{app:ext:compare}, \ref{cond:equal} implies our main identifiability Condition~\ref{cond:ident:main}.

\begin{itemize}
    \item ``Mod'' model (MOD): $X_k= (S_k\mod2)^{Z_k}\times (1 - (S_k\mod2))^{1-Z_k}$ where $S_k = \sum_{\ell\in\pa(k)} X_{\ell}$ with $Z_k\sim \BernoulliDist(0.2)$
    \item Additive model (ADD): $X_k = \sum_{\ell\in\pa(k)}X_\ell + Z_k$ with $Z_k\sim\BernoulliDist(0.2)$
\end{itemize}

Figure~\ref{fig:exp:main} (left, right) illustrates the performance in terms of structural Hamming distance (SHD) between the true graph and the estimated graph. As expected, our algorithm successfully recovers the underlying DAG $\gr$ and performs comparably to PC and GES, which also perform quite well.
We stress that the current experiments are used for simple illustration thus not well-optimized compared to existing algorithms or fully exploited the main condition either. The relatively good performance of PC/GES partially comes from the fact that our synthetic models are all in fact faithful, see Appendix~\ref{app:unfaith} for further discussion.

\begin{figure}[t]%
\includegraphics[width=1.\textwidth]{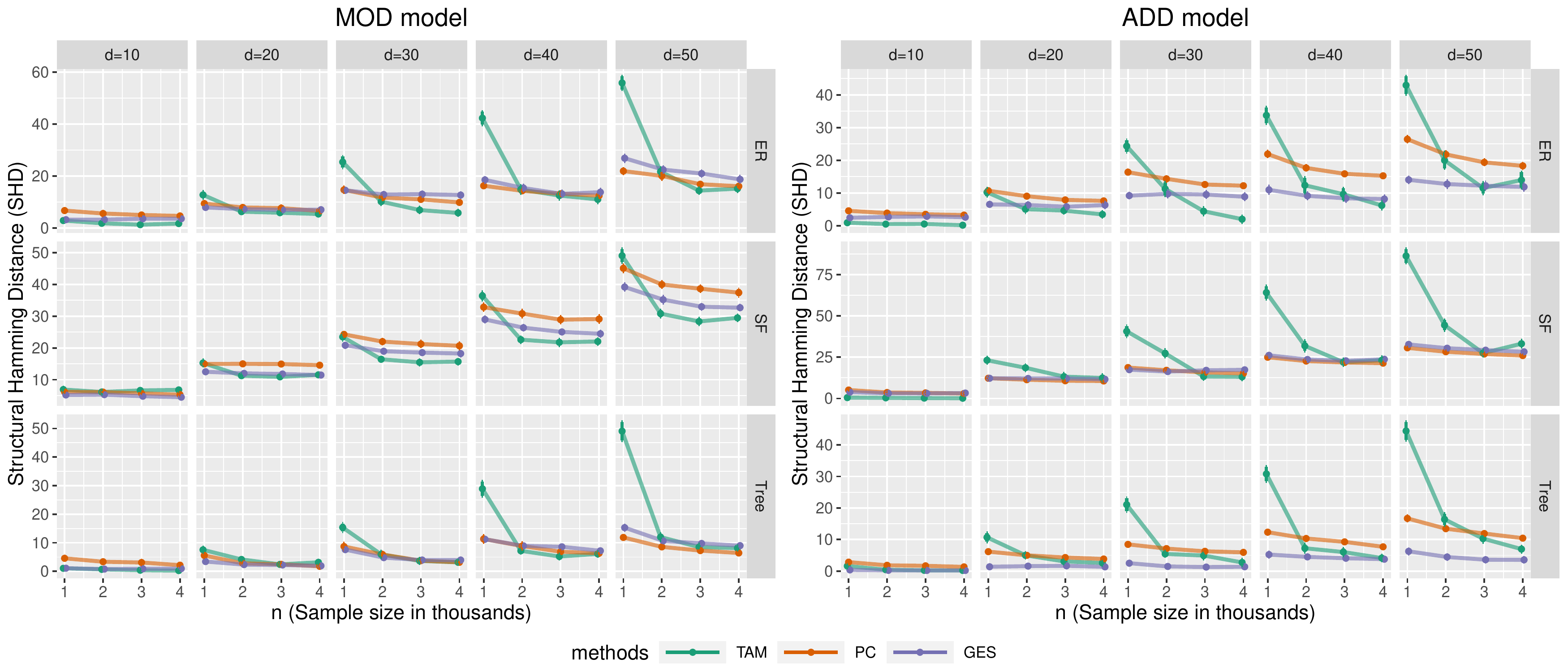}
\caption{SHD vs sample size $n$ for different dimensions and graph types. Left panel is for MOD model; Right panel is for ADD model.}
\label{fig:exp:main}
\end{figure}

\section{Conclusion}
\label{sec:conc}

The main goal of this paper has been to better understand the underlying assumptions required for DAG models to be estimated from data. To this end, we have provided a new identifiability result along with a learning algorithm, which turns out to generalize existing ones, and analyzed a greedy, local search algorithm for discovering Markov boundaries. This local search algorithm can be used to provably learn the structure of a minimal I-map in polynomial time and sample complexity as long as the Markov boundaries are not too large. 
Nonetheless, there are many interesting directions for future work. Perhaps the most obvious is relaxing the logarithmic dependence on $d$ in $\mbsize$. 
It would also interesting to investigate lower bounds on the sample complexity of this model, as well as additional identifiability conditions.

\subsection*{Acknowledgements}
We thank the anonymous reviewers for their helpful comments in improving the manuscript.
B.A. was supported by NSF IIS-1956330, NIH R01GM140467, and the Robert H. Topel Faculty Research Fund at the University of Chicago Booth School of Business.
All statements made are solely due to the authors and have not been endorsed by the NSF.

\bibliography{foo-arxiv}
\bibliographystyle{abbrvnat}

\appendix 

\section{Complete algorithm description}
\label{app:compalg}

For completeness and reproducibility, the full TAM algorithm combining Algorithms~\ref{alg:uneq:short} and~\ref{alg:pps} is detailed in Algorithm~\ref{alg:uneq:withpps}.

\begin{algorithm}[!h]
\caption{TAM algorithm for learning DAGs}
\label{alg:uneq:withpps}
\textbf{Input:} $X=(X_1,\ldots,X_d)$, $\unentthresh$, $\ppsthresh$\\
\textbf{Output:} $\widehat{\gr}$.
\begin{enumerate}
\item Initialize empty graph $\widehat{\gr}=\emptyset$ and $j=0$.
\item Set layer $\wh{\layer}_0=\emptyset$, let $\wh{\anc}_{j}=\cup_{t=0}^j \wh{\layer}_t$.
\item While $V\setminus \widehat{\anc}_j\ne\emptyset$:
\begin{enumerate}
\item For $k\notin \wh{\anc}_{j}$, apply PPS($X,k,\wh{\anc}_{j},\ppsthresh$) (See Algorithm~\ref{alg:pps}) to obtain the estimated Markov boundary $\wh{\mkvbdy}_{jk}$ along with an estimate $\wh{h}_{jk}$ of the corresponding conditional entropy $\ent(X_k\given \wh{\anc}_{j})$. 
\item Initialize $\wh{\layer}_{j+1} = \emptyset$, $\wh{S}_{j+1} = \emptyset$. Sort $\wh{h}_{jk}$ in ascending order and let $\wh{\tau}^{(0)}$ be the corresponding permutation of $V\setminus \anc_j$. 
\item For $\ell\in 0,1,2,\ldots$ until $|\wh{\tau}^{(\ell)}|=0$: \textbf{[TAM step]}
\begin{enumerate}
\item  $\wh{\layer}_{j+1}=\wh{\layer}_{j+1}\cup \{\wh{\tau}^{(\ell)}_1\}$.
\item For $k\notin \wh{\anc}_{j}\cup \wh{\layer}_{j+1} \cup \wh{S}_{j+1}$, estimate $\mi(X_k; \wh{\tau}^{(\ell)}_1 \given \wh{\mkvbdy}_{jk})$ by some estimator $\wh{I}^{(\ell)}_{jk}$
\item Set $\wh{S}_{j+1} = \wh{S}_{j+1}\cup \{k: \wh{I}^{(\ell)}_{jk} \ge \unentthresh\}$.
\item $\wh{\tau}^{(\ell+1)}=\wh{\tau}^{(\ell)}\setminus \left(\wh{\layer}_{j+1} \cup \wh{S}_{j+1}\right)$
\end{enumerate}
\item\label{alg:uneq:withpps:pa} For $k\in \wh{\layer}_{j+1}$, set $\pa_{\wh{\gr}}(k)=\wh{\mkvbdy}_{jk}$.
\item Update $j=j+1$.
\end{enumerate}
\item Return $\widehat{\gr}$.
\end{enumerate}
\end{algorithm}

\section{Graphical model background}
\label{app:gm}

In this appendix, we recall some basic facts regarding graphical models that are used throughout the proofs. This section will also help to explain the importance of the positivity assumption on $\truepr$, as well as the concept of faithfulness. For more background on graphical models, see \citet{lauritzen1996,koller2009}.

\paragraph{Uniqueness of Markov boundaries}
The \emph{Markov blanket} of a node $X_{k}$ relative to some subset $S\subset V$ is any subset $m\subset S$ such that $X_{k}\indep (S\setminus m)\given m$. A Markov boundary  is a minimal Markov blanket, i.e. a Markov blanket $m$ such that no proper subset of $m'\subsetneq m$ satisfies $X_{k}\indep (S\setminus m') \given m'$. Neither the Markov blanket nor the Markov boundary are unique in general. A key fact regarding Markov boundaries is that when $\truepr$ is strictly positive, they are unique. Recall that we denote the Markov boundary of $X_{k}$ relative to $S$ by $\MB(X_{k};S)$.
\begin{lemma}
\label{lem:mb:unique}
If $\truepr(X=x)>0$ for all $x\in\{0,1\}^{d}$, then for any $S\subset V$, the Markov boundary $\MB(X_{k};S)$ exists and is unique.
\end{lemma}
For a direct proof, see Proposition~3.1.3 in \citet{drton2009}. This lemma remains true if $\truepr$ is replaced by a (strictly positive) density function. More generally, Markov boundaries are unique as long as the \emph{intersection property} of conditional independence holds in $\truepr$ \citep{dawid1980,peters2015intersection}.

\paragraph{Minimal I-maps and orderings}

A minimal I-map of $\truepr$ is any DAG $\gr=(V,E)$ such that the following conditions hold:
\begin{enumerate}
\item $\truepr$ factorizes over $\gr$, i.e. \eqref{eq:defn:bn} holds, and
\item If any edge is removed from $E$, then \eqref{eq:defn:bn} is violated.
\end{enumerate}
In general, minimal I-maps are not unique. Given an ordering $\ord$ of the variables, a minimal I-map can be constructed as follows \citep[see e.g.,][\sec3.4.1]{koller2009}: For each $k$, define $\pa(k)$ to be $\MB(X_{k};\prec_{k})$, where $\prec_{k}:=\{j:X_{j}\ord X_{k}\}$. This procedure is well-defined as long as $\MB(X_{k};\prec_{k})$ is unique, which is guaranteed by Lemma~\ref{lem:mb:unique}. 
An important consequence of this procedure is that once the layer decomposition of a minimal I-map $\gr$ is known, the full DAG $\gr$ can be recovered by performing local search. To see this, recall that the layers $\layer_{j}$ define canonical ancestral sets $\anc_{j}$ and replace $\MB(X_{k};\prec_{k})$ above with $\MB(X_{k};\anc_{j})$, where $j$ is the largest index such that $X_{k}\notin\anc_{j}$.

\paragraph{Faithfulness and $d$-separation}
Throughout the proofs, we make use of the concept of $d$-separation defined below. Although faithfulness is never assumed, it is useful to recall its definition for completeness.

A \emph{trail} in a directed graph $\gr$ is any sequence of distinct nodes $X_{i_{1}},\ldots,X_{i_{\ell}}$ such that there is an edge between $X_{i_{m}}$ and $X_{i_{m+1}}$. The orientation of the edges does not matter. For example, $X\leftarrow Y\rightarrow Z$ would be a valid trail. 
A trail of the form $X_{i_{m-1}}\rightarrow X_{i_{m}}\leftarrow X_{i_{m+1}}$ is called a \emph{$v$-structure}.
We say a trail $t$ is \emph{active} given another set $C$ if 
(a) $\de(X_{i_{m}})\cap C\ne\emptyset$ for every $v$-structure $X_{i_{m-1}}\rightarrow X_{i_{m}}\leftarrow X_{i_{m+1}}$ in $t$
and (b) no other node in $t$ is in $C$. In other words, $t\cap C$ consists only of central nodes in some $v$-structure contained entirely in $t$.
\begin{definition}
Let $A,B,C$ be three sets of nodes in $\gr$. We say that $A$ and $B$ are $d$-separated by $C$ if there is no active trail between any node $a\in A$ and $b\in B$ given $C$.
\end{definition}
An important consequence of \eqref{eq:defn:bn} is the following: If $\gr$ satisfies \eqref{eq:defn:bn} for some $\truepr$ and $A$ and $B$ are $d$-separated by $C$ in $\gr$, then $A\indep B\given C$ in $\truepr$ (\citealp{koller2009}, Theorems~3.1, 3.2, or \sec3.2.2 in \citealp{lauritzen1996}).

Thus, if $\gr$ is a BN of $\truepr$, then $d$-separation can be used to read off a subset of the conditional independence relations in $\truepr$. Whenever the reverse implication holds---i.e. conditional independence in $\truepr$ implies $d$-separation in $\gr$, we say that $\truepr$ is \emph{faithful} to $\gr$. This condition does not hold in general, not even for minimal I-maps.

\begin{remark}
Faithfulness is a standard assumption in the literature on BNs. Assuming $\gr$ is faithful to $\truepr$ ensures that the \emph{Markov equivalence class} of $\truepr$ is identified, however, this is not the same as identifying $\gr$. More precisely, faithfulness identifies a CPDAG, which is a partially directed graph that encodes the set of conditional independence relationships shared by every DAG in the Markov equivalence class. This can be a strong assumption, especially with finite samples \citep{uhler2013}. By contrast, our approach is to circumvent faithfulness and impose assumptions that identify a bona-fide DAG $\gr$. See Appendix~\ref{app:unfaith} for an explicit example where Condition~\ref{cond:ident:main} identifies an unfaithful DAG.
\end{remark}

\section{Extensions and further examples}
\label{app:ext}

The results in Sections~\ref{sec:ident}-\ref{sec:dag} make several assumptions that are not strictly necessary. In this appendix, we briefly outline how these assumptions can be relaxed. 
In Appendix~\ref{app:mi} we show how the positivity assumption can be replaced by a slightly weaker nondegeneracy condition.
In Appendix~\ref{app:ext:general}-\ref{app:ext:compare} we discuss how Condition~\ref{cond:ident:main} can be relaxed and how it is compared with existing identifiability results. Then we discuss examples and extensions of Condition~\ref{cond:pps} in Appendices~\ref{app:pps:eg}-\ref{app:ext:pps}. Finally, in Appendix~\ref{sec:ext:gen} we discuss extensions to more general distributions.

\subsection{Positivity and a nondegeneracy condition}\label{app:mi}

Throughout the paper, we have assumed that $\truepr$ is strictly positive. In fact, all of the results will go through under the following slightly weaker condition:
\begin{condition}[Nondegeneracy]
\label{cond:mi}
$I(\rv_{k};\pa(k)\given \anc)>0$ for any ancestral set $\anc\subset[d]$ such that $\pa(k)\setminus\anc\ne\emptyset$.
\end{condition}
This condition implies that there is still some information between $X_{k}$ and $\pa(k)$, even after learning everything about $\anc$. This is quite reasonable: If $\pa(k)\setminus\anc\ne\emptyset$ is nonempty, then there is still at least one parent of $X_{k}$ unaccounted for after conditioning on $\anc$. This ``missing parent'' accounts for the ``missing mutual information'' that makes this quantity positive. Seemingly reasonable, there are some degenerate cases where Condition \ref{cond:mi} may not hold. The following lemma makes the characterization of nondegeneracy more precise:
\begin{lemma}
\label{lem:mi}
Suppose that for any two disjoint subsets $A,B\subset X$, $\truepr(A\given B)\notin\{0,1\}$. Then Condition \ref{cond:mi} holds for any minimal I-map of $\truepr$.
\end{lemma}
This lemma shows that as long as the dependencies implied by $\gr$ are non-deterministic, Condition \ref{cond:mi} will always be satisfied. In fact, this is guaranteed by the positivity of $\truepr$, which we have assumed already:
\begin{corollary}
\label{cor:mi:pos}
If $\truepr$ is strictly positive, then Condition \ref{cond:mi} holds for any minimal I-map of $\truepr$.
\end{corollary}

Before proving Lemma~\ref{lem:mi}, we first show how Corollary~\ref{cor:mi:pos} follows as a consequence. Indeed, this follows immediately from Lemma~\ref{lem:mi} and the following lemma:
\begin{lemma}\label{lem:ndtm}
If $\truepr(X=x)>0$ for all $x\in\{0,1\}^{d}$, then for any two disjoint subsets $A,B\subset X$, $\truepr(A\given B)\notin\{0,1\}$.
\end{lemma}
\begin{proof}
For any $a,b$, Bayes' rule implies $\truepr(A=a\given B=b)>0$. Now we show $\truepr(A=a\given B=b)\ne1$. Suppose the contrary, then
\[
\truepr(A=a,B=b) = \truepr(B=b) = \sum_{a'}\truepr(A=a',B=b)
\implies 
\sum_{a'\ne a}\truepr(A=a',B=b) = 0,
\]
which is contradictory to $\truepr(A=a, B=b)>0$ and completes the proof.
\end{proof}

\noindent
Now we prove Lemma~\ref{lem:mi}.
\medskip
\begin{proof}(Proof of Lemma~\ref{lem:mi})
We proceed by contradiction. 
Suppose $I(X_{k};\pa(k)\given \anc))=0$, then 
\[
X_k \indep \pa(k)\given \anc
\]
Let $\pa_1(k)=\pa(k)\cap A$, $\pa_2(k)=\pa(k)\setminus A$, and $\anc' = \anc\setminus\pa_1(k)$, so that $\pa(k)=\pa_1(k)\cup \pa_2(k)$, $\anc = \pa_1(k)\cup \anc'$, and $\pa_2(k)\cap \anc=\emptyset$, $\pa_2(k)\ne \emptyset$. Therefore,
\[
\truepr(X_k\given \pa_1(k),\anc')\truepr(\pa_1(k),\pa_2(k)\given \pa_1(k),\anc')=\truepr(X_k, \pa_1(k),\pa_2(k)\given \pa_1(k),\anc').
\]
Since 
\begin{align*}
\truepr(\pa_2(k)=y_2\given \pa_1(k)=y_1',\anc'=a') = &
\sum_{y_1}\truepr(\pa_1(k)=y_1,\pa_2(k)=y_2\given \pa_1(k)=y_1',\anc'=a')\\
=&\sum_{y_1=y_1'}\truepr(\pa_1(k)=y_1,\pa_2(k)=y_2\given \pa_1(k)=y_1',\anc'=a')\\ %
=&\truepr(\pa_1(k)=y_1',\pa_2(k)=y_2\given \pa_1(k)=y_1',\anc'=a').
\end{align*}
Thus we can simplify to
\[
\truepr(X_k\given \pa_1(k),\anc')\truepr(\pa_2(k)\given \pa_1(k),\anc')=\truepr(X_k, \pa_2(k)\given \pa_1(k),\anc')
\]
which amounts to
\[
X_k\indep \pa_2(k) \given (\anc',\pa_1(k)).
\]
Combined with the Markov property $X_k\indep \anc'\given (\pa_1(k),\pa_2(k))$ (i.e. since $A'\subset\nd(k)$), we have
\begin{align*}
    \truepr(X_k,\pa_2(k),\anc'\given \pa_1(k)) & = \truepr(X_k,\pa_2(k)\given \anc',\pa_1(k))\truepr(\anc'\given \pa_1(k))\\
    &= \truepr(X_k\given \anc',\pa_1(k))\truepr(\pa_2(k), \anc'\given \pa_1(k))\\
    & = \truepr(X_k,\anc'\given \pa_1(k),\pa_2(k))\truepr(\pa_2(k)\given \pa_1(k))\\
    &= \truepr(X_k\given \pa_2(k),\pa_1(k))\truepr(\pa_2(k), \anc'\given \pa_1(k)).
\end{align*}
By Condition~\ref{cond:mi}, $\truepr(\pa_2(k), \anc'\given \pa_1(k))\notin\{0,1\}$, so that
\[
\truepr(X_k\given \anc',\pa_1(k)) = \truepr(X_k\given \pa_2(k),\pa_1(k))
\]
holds for different combinations of ($\anc',\pa_2(k)$). These are two functions of $\anc'$ and $\pa_2(k)$ given $\pa_1(k)$, and are equal for all possible combinations of ($\anc',\pa_2(k)$), i.e. it is independent of what values they take on. It follows that 
\[
\truepr(X_k\given \pa_1(k)) = \truepr(X_k\given \pa_2(k),\pa_1(k)).
\]
Since $\gr$ is an I-map, the joint probability factorizes over it
\begin{align*}
\truepr(X_1,\cdots,X_d)&=\prod_{\ell=1}^d \truepr(X_\ell\given \pa(\ell))\\
&= \truepr(X_k\given \pa_1(k),\pa_2(k)) \prod_{\ell\ne k} \truepr(X_\ell\given \pa(\ell))\\ 
&= \truepr(X_k\given \pa_1(k)) \prod_{\ell\ne k} \truepr(X_\ell\given \pa(\ell)).
\end{align*}
Therefore we can remove the edges from $\pa_2(k)$ to $X_k$, which contradicts the minimality of $\gr$. The proof is complete.
\end{proof}

\subsection{More general version of Condition~\ref{cond:ident:main}}\label{app:ext:general}
According to Lemma~\ref{lem:poly}, one sufficient condition for \ref{cond:uneq:mi:short} is to have no path-cancellation between $X_i$ and $X_k$. This can be further relaxed by following Theorem~\ref{thm:ident:unequal2} 
and Algorithm~\ref{alg:uneq}.
For any ancestor $X_i\in \an_j(k)$, denote the subset $\anprime \subseteq \an_j(k)\setminus X_i$ to be the ancestors with smaller conditional entropies than $X_i$, namely,
\begin{align*}
    \anprime
    := \{X_\ell \in \an_j(k)\setminus X_i : \ent(X_\ell\given \anc_{j}) \le \ent(X_i\given \anc_{j}) \}
\end{align*}
\begin{theorem}
\label{thm:ident:unequal2}
For each $X_{k}\in V$ with $\layer(X_k)\ge 2$, and for each $j=0,\cdots,\layer(X_{k})-2$,
there exists $X_i\in \an_j(k)$, which we refer as important ancestors $\imp_j(k)$, such that following two conditions hold:
\begin{enumerate}[label=(U\arabic*)., ref=(U\arabic*), wide = 0pt]
\item $\ent(X_i\given \anc_{j}) < \ent(X_k \given \anc_{j})$\label{cond:uneq}
\item $I(X_k; (\anprime,X_i) \given \anc_{j}) > 0$\label{cond:uneq:mi},
\end{enumerate}
Then $\gr$ is identifiable from $\truepr$.
\end{theorem}
\begin{remark}
\label{rem:gen}
For any node $X_k$ not in current layer $\layer_{j+1}$, \ref{cond:uneq} is exactly the same as \ref{cond:uneq:short}. 
\ref{cond:uneq:mi} requires its important ancestor together with other ancestors with smaller entropy in $\layer_{j+1}$ contribute to positive mutual information with $X_k$. Note that $\anprime$ can be empty, in which case \ref{cond:uneq:mi} reduces to \ref{cond:uneq:mi:short}. If this is applied to the equal entropy case in Condition~\ref{cond:equal}, \ref{cond:uneq:mi} is relaxed to 
\[
I(X_k;\layer_{j+1}\given \anc_{j})>0
\]
which requires the entropy of nodes not in $\layer_{j+1}$ to be conditional dependent with all nodes in $\layer_{j+1}$.
\end{remark}

\begin{algorithm}[t]
\caption{TAM algorithm (general version)}
\label{alg:uneq}
\textbf{Input:} $X=(X_1,\ldots,X_d)$, $\unentthresh$ \\
\textbf{Output:} $\wh{\layer}=(\wh{\layer}_1,\ldots,\wh{\layer}_{\wh{\depth}})$.
\begin{enumerate}
\item Initialize $\wh{\layer}_0=\emptyset$, let $\wh{\anc}_{j}=\cup_{t=0}^j \wh{\layer}_t$
\item For $j\in 0,1,2,\ldots$:
\begin{enumerate}
\item For $k\notin \wh{\anc}_{j}$, estimate conditional entropy $\ent(X_k\given \wh{\anc}_{j})$ by some estimator $\wh{h}_{jk}$. 
\item Initialize $\wh{\layer}_{j+1} = \emptyset$, $\wh{S}_{j+1} = \emptyset$. Sort $\wh{h}_{jk}$ in ascending order and let $\wh{\tau}^{(0)}$ be the corresponding permutation of $V\setminus \anc_j$. 
\item For $\ell\in 0,1,2,\ldots$ until $|\wh{\tau}^{(\ell)}|=0$: \textbf{[TAM step]}
\begin{enumerate}
\item Let $\wh{\layer}_{j+1}=\wh{\layer}_{j+1}\cup \{\wh{\tau}^{(\ell)}_1\}$.
\item For $k\notin \wh{\anc}_{j}\cup \wh{\layer}_{j+1} \cup \wh{S}_{j+1}$, estimate $\mi(X_k;\widehat{\layer}_{j+1}\given \wh{\anc}_{j})$ by some estimator $\wh{\mi}_{jk}^{(\ell)}$
\item Set $\wh{S}_{j+1} = \wh{S}_{j+1}\cup \{k: \wh{\mi}_{jk}^{(\ell)} \ge \unentthresh\}$.
\item $\wh{\tau}^{(\ell+1)}=\wh{\tau}^{(\ell)}\setminus \left(\widehat{\layer}_{j+1}\cup \widehat{S}_{j+1}\right)$
\end{enumerate}
\end{enumerate}
\item Return $\wh{\layer}=(\wh{\layer}_1,\ldots,\wh{\layer}_{\wh{\depth}})$.
\end{enumerate}
\end{algorithm}
Algorithm~\ref{alg:uneq} differs with Algorithm~\ref{alg:uneq:short} in only one step. When testing independence in the TAM step, instead of estimating $\mi(X_k;\widehat{\tau}_1^{(\ell)}\given\widehat{\anc}_j)$, we choose to estimate $\mi(X_k;\widehat{\layer}_{j+1}\given\widehat{\anc}_j)$ to take advantage of $\anprime$ to detect $X_k\notin\layer_{j+1}$.
Since Theorem~\ref{thm:ident:unequal2:short} and Algorithm~\ref{alg:uneq:short} are special cases of Theorem~\ref{thm:ident:unequal2} and Algorithm~\ref{alg:uneq}, we only show the proof of the later (more general) theorem and the correctness of the later algorithm, which are shown in Appendix~\ref{app:uneq2}.

\subsection{Comparison with ``equal entropy'' condition}\label{app:ext:compare}
The ``equal entropy'' Condition~\ref{cond:equal} has a straightforward relaxation which follows directly along similar lines as \citet{ghoshal2017ident}. For completeness, we quote this result below; the proof is identical to this prior work and hence omitted. Denote $h_k=\ent(X_k\given \pa(k))$.
\begin{condition}\label{cond:uneq1}
There exists a topological ordering $\tau$ such that for all $j\in {[d]}$ and $\ell \in \tau_{[j+1:d]}$, the following holds: 
\begin{enumerate}
\item If $k=\tau_j$ and $\ell$ are not in the same layer, then
\begin{align}
\label{eq:thm:ident:unequal1}
    h_k &< h_\ell + I(X_\ell; \pa(\ell)\setminus \tau_{[1:j-1]} \given X_{\tau_{[1:j-1]}}).
\end{align}
\item If $k$ and $\ell$ are in the same layer, then either $h_k=h_\ell$ or \eqref{eq:thm:ident:unequal1} holds.
\end{enumerate}
\end{condition}
\begin{theorem}
\label{thm:ident:unequal1}
If Condition~\ref{cond:uneq1} holds for some ordering $\tau$, then $\tau$ is identifiable from $\truepr$.
\end{theorem}
We can interpret Condition~\ref{cond:uneq1} in following two ways. 
As long as the conditional mutual information on the right side of \eqref{eq:thm:ident:unequal1} is positive, then the topological sort can be recovered by minimizing conditional entropies, much like the equal variance algorithm.
On the other hand, compared to Condition~\ref{cond:equal} which requires $h_{k}=h_{\ell}$, 
the conditional mutual information between child and parent nodes on the RHS is a bound on the difference $h_{k}-h_{\ell}$. Thus, Condition~\ref{cond:uneq1} can be interpreted as a relaxation of Condition~\ref{cond:equal} in which the ``violations'' of equality are controlled by this conditional mutual information.

\begin{lemma}
\label{lem:uneq2main}
Assuming \ref{cond:uneq:mi:short}, Condition~\ref{cond:uneq1} implies Condition~\ref{cond:ident:main}.
\end{lemma}

\begin{proof}
Suppose $X_i=X_{\tau_s}\in \an_j(k)$ has the closest position to $X_k$ in terms of ordering $\tau$, then $X_{\tau_{[1:s-1]}}$ contains all the other ancestors of $X_k$ in $\layer_{j+1}$ except for $X_i$. Furthermore, $X_{\tau_{[1:s-1]}}$ also contains the ancestors of $\an_j(k)$ such that $P(X_k\given \anc_j) = P(X_k\given \text{ancestors of }\an_j(k))$. Therefore, we have \ref{cond:uneq:short}:
\begin{align*}
     \ent(X_i\given \anc_{j}) &= \ent(X_i\given X_{\tau_{[1:s-1]}}) \\
     &< \ent(X_k\given X_{\tau_{[1:s-1]}}) \\
     &= \ent(X_k \given \anc_{j}, \an_j(k)\setminus X_i) \\
     &\le \ent(X_k \given \anc_{j})
\end{align*}
The first inequality is by \eqref{eq:thm:ident:unequal1}, the second one uses the increasing property of conditional entropy when conditioning set is enlarged.
\end{proof}

\begin{remark}
Lemma~\ref{lem:uneq2main} implies Lemma~\ref{lem:equal2main}. In fact,
since Condition~\ref{cond:equal} implies Condition~\ref{cond:uneq1}, we have
\begin{align*}
    \text{Condition~\ref{cond:equal}}
    \implies
    \text{Condition~\ref{cond:uneq1}}
    \implies
    \text{Condition~\ref{cond:ident:main}}.
\end{align*}
\end{remark}

Finally, we conclude with an example to illustrate the difference between Condition~\ref{cond:ident:main} and Condition~\ref{cond:uneq1}.
\begin{example}\label{eg:uneq}
Suppose the graph is 
\begin{align*}
 \ \ \ \ & X_2 \\
 \ \ \nearrow & \ \ \searrow \\
 X_1 \ \ \ \ & \ \ \ \ \ \ \ \ X_4  \\
 \ \ \searrow & \ \ \nearrow \\
 \ \ \ \ & X_3 
\end{align*}
where
\begin{align*}
X_1=Z_1,   & \ \ \ \ Z_1\sim Ber(0.2), \\
X_2=-X_1+Z_2 = -Z_1+Z_2, & \ \ \ \ Z_2\sim Ber(0.1), \\
X_3=X_1+Z_3 = Z_1+Z_3, & \ \ \ \ Z_3\sim Ber(0.2).
\end{align*}
For $X_4$, we consider two models:
\begin{enumerate}[label=(M\arabic*)]
\item\label{model:ident:unequal2} $X_4 = X_2 + Ber(\sigma(\epsilon X_3 + \beta_0))$
\item\label{model:ident:unequal1} $X_4 = X_2+X_3+Z_4 = Z_2+Z_3+Z_4,\ \ \ \ Z_4\sim Ber(0.1)$,
\end{enumerate}
where $\sigma(x)=1/(1+\exp(-x))$. It is straightforward to check that each of these models does not satisfy equalities in Condition~\ref{cond:equal}, but satisfies either Condition~\ref{cond:ident:main} or Condition~\ref{cond:uneq1}. Let's first look at the first three variables:
\begin{align*}
\ent(X_1)=h(0.2)\approx 0.500
\qquad\ent(X_2) &\approx 0.733
\qquad\ent(X_3)\approx 0.778 \\
\ent(X_2\given X_1)=h(0.1)\approx 0.325<\ent(X_2)
\quad&\ent(X_3\given X_1)=h(0.2)\approx 0.500<\ent(X_3)
\end{align*}
So the subgraph on $(X_1,X_2,X_3)$ satisfies both Condition~\ref{cond:ident:main} and Conditions~\ref{cond:uneq1}. Now consider the last node: For the first model \ref{model:ident:unequal2}, $\beta_0=\sigma^{-1}(0.1)=\log\frac{0.1}{1-0.1}\approx -0.219$. 
Choose $\epsilon$ small enough so that $\sigma(\epsilon X_3 + \beta_0)\approx 0.1$, e.g. if $\epsilon=0.01$ then $\sigma(\epsilon X_3 + \beta_0)\in [0.1,0.102]$ when $X_3\in\{0,1,2\}$.
By doing so,
\begin{align*}
&\truepr(X_4=-1)=0.2\times 0.9\times 0.9 && \truepr(X_4=-X_1+0\given X_1)=0.9\times 0.9\\
&\truepr(X_4=0)=0.8\times 0.9\times 0.9+0.2\times 0.9\times 0.1\times 2 && \truepr(X_4=-X_1+1\given X_1)=0.9\times 0.1\times 2\\ 
&\truepr(X_4=1)=0.8\times 0.9\times 0.1\times 2 + 0.2\times 0.1\times 0.1 && \truepr(X_4=-X_1+2\given X_1)=0.1\times 0.1\\
&\truepr(X_4=2)=0.8\times 0.1\times0.1 && \ent(X_4\given X_1)\approx 0.525 < \ent(X_4)\\
&\ent(X_4) \approx 0.87 && \ent(X_4\given X_1,X_2)\approx h(0.1) \approx 0.325 \\
& && < \ent(X_3\given X_1,X_2)<\ent(X_4\given X_1)
\end{align*}
Therefore \ref{model:ident:unequal2} satisfies \ref{cond:uneq:short} and \ref{cond:uneq:mi:short} but not Condition~\ref{cond:uneq1}, since the order between $X_3$ and $X_4$ is flipped. On the other hand, the second model \ref{model:ident:unequal1}, where $X_1\indep X_4$, $I(X_1;X_4)=0$, thus violates \ref{cond:uneq:mi:short} and will incorporate $X_4$ into the first layer. But Condition~\ref{cond:uneq1} still holds. Note that \ref{model:ident:unequal2} is also an example of path-cancellation, which is discussed in Lemma~\ref{lem:poly}.

\end{example}

\subsection{Examples of PPS condition}\label{app:pps:eg}
In this section, we illustrate some examples of DAGs satisfying Condition~\ref{cond:pps} for which there may exist multiple directed paths between two nodes. This shows that the poly-forest assumption in Theorem~\ref{thm:poly} is sufficient but not necessary for Condition~\ref{cond:pps} to hold.

Suppose we have a simple graph on 4 nodes: $V=(Z,X_1,X_2,Y)$. The edges are 
\begin{align*}
    &Z\rightarrow X_1 \rightarrow Y \\
    &Z\rightarrow X_2 \rightarrow Y 
\end{align*}
So there are 2 paths from $Z$ to $Y$. The basic idea is to let $X_1,X_2$ have the same effect on $Y$, while $Z$ has opposite effects on $X_1,X_2$. Then when $Z$ is changed, the distribution of $Y$ is similar. Assuming a logistic model for each conditional probability distribution on this graph:
\begin{align*}
    \truepr(X_i=1\given Z=z) &= \sigma(\beta_i z + \alpha_i) \ \ \ \ i=1,2\\
    \truepr(Y=1\given X_1=x_1,X_2=x_2) &= \sigma(\beta_0(x_1+x_2)+\alpha_0)
\end{align*}
where $\sigma(x) = 1/({1+\exp(-x)})$. It suffices to choose parameters so that $\mi(Z;Y)=0$, i.e. $\truepr(Y=1\given Z=1)= \truepr(Y=1\given Z=0)$. We have
\begin{align*}
    \truepr(Y=1\given Z=1) =& \sigma(\beta_1+\alpha_1)\sigma(\beta_2+\alpha_2)\sigma(2\beta_0+\alpha_0)\\
    +&\sigma(\beta_1+\alpha_1)\sigma(-\beta_2-\alpha_2)\sigma(\beta_0+\alpha_0)\\
    +&\sigma(-\beta_1-\alpha_1)\sigma(\beta_2+\alpha_2)\sigma(\beta_0+\alpha_0)\\
    +&\sigma(-\beta_1-\alpha_1)\sigma(-\beta_2-\alpha_2)\sigma(\alpha_0)
\end{align*}
and 
\begin{align*}
    \truepr(Y=1\given Z=0) =& \sigma(\alpha_1)\sigma(\alpha_2)\sigma(2\beta_0+\alpha_0)\\
    +&\sigma(\alpha_1)\sigma(-\alpha_2)\sigma(\beta_0+\alpha_0)\\
    +&\sigma(-\alpha_1)\sigma(\alpha_2)\sigma(\beta_0+\alpha_0)\\
    +&\sigma(-\alpha_1)\sigma(-\alpha_2)\sigma(\alpha_0),
\end{align*}
so to make them equal it suffices to have
\begin{align*}
    &\beta_1+\alpha_1 = \alpha_2\\
    &\beta_2+\alpha_2 = \alpha_1
\end{align*}
This implies $\beta_1+\beta_2=0$, the effects from $Z$ are neutralized. 

By a similar argument, this example can be generalized to $n$ nodes in the middle layer, i.e. $V=(Z,X_1,\ldots,X_n,Y)$ and $Z\to X_i\to Y$ for each $i$. In this case, it suffices to have $\sum_{i=1}^n\beta_i=0$. This can be further generalized to $n$ paths from $Z$ to $Y$ with $m_1,\cdots,m_n$ nodes on each path, 
\begin{align*}
    & Z\rightarrow X_{11} \rightarrow X_{12} \rightarrow \cdots \rightarrow X_{1m_1} \rightarrow Y \\
    & Z\rightarrow X_{21} \rightarrow X_{22} \rightarrow \cdots \rightarrow X_{2m_2} \rightarrow Y \\
    &\cdots \\
    & Z\rightarrow X_{n1} \rightarrow X_{n2} \rightarrow \cdots \rightarrow X_{nm_n} \rightarrow Y
\end{align*}
The requirement now is that there exists one node in each path
\begin{align*}
    &X_{1i_1} \in \{X_{11},X_{12},\cdots,X_{1n_1}\}\\
    &X_{2i_2} \in \{X_{21},X_{22},\cdots,X_{2n_2}\}\\
    & \cdots \\
    &X_{ni_n} \in \{X_{n1},X_{n2},\cdots,X_{nm_n}\}
\end{align*}
whose coefficients satisfy the equation we derived above. 

\begin{remark}
Although the previous examples assume exact independence between $Y$ and $Z$ (i.e. $\mi(Y;Z)=0$),
this argument can be generalized as long as $\mi(Y;Z)<\inf_k \mi(Y;X_k)$.
\end{remark}

\subsection{Relaxing Condition~\ref{cond:pps}}\label{app:ext:pps}

In the previous section, we constructed examples of DAGs satisfying Condition~\ref{cond:pps} that were not poly-forests. This confirms that our results apply more broadly than Theorem~\ref{thm:poly} would suggest. In this section, we show that this condition can be eliminated altogether by sacrificing the sample complexity.
Here we simply replace the PPS algorithm with direct estimation of $\ent(X_k\given\wh{\anc}_{j})$. Then we can apply the backward phase of the IAMB algorithm \citep{tsamardinos2003algorithms} to infer the parents of each node in $\wh{\layer}_{j+1}$ from $\wh{\anc}_{j}$ in order to learn the whole graph. Algorithm \ref{alg:iamb} describes the backward phase of IAMB algorithm tailored for Algorithm~\ref{alg:uneq:short}, which essentially uses $\wh{\anc}_{j}$ as a candidate set and conditional mutual information (with a parameter $\ppsthresh$) as an independence test.
\begin{algorithm}[t]
\caption{Backward phase of IAMB algorithm}
\label{alg:iamb}
\textbf{Input:} $X=(X_1,\cdots,X_d)$, $k$, $\wh{\anc}_j$, $\ppsthresh$ \\
\textbf{Output:} Markov boundary estimate $\widehat{\mkvbdy}_{jk}$
\begin{enumerate}
\item Initialize $\widehat{\mkvbdy}_{jk} = \wh{\anc}_j$
\item $\widehat{\mkvbdy}_{jk} = \widehat{\mkvbdy}_{jk}\setminus \big\{X_\ell\in \widehat{\mkvbdy}_{jk}: \wh{I}(X_\ell;X_k \given \widehat{\mkvbdy}_{jk}\setminus X_\ell) < \ppsthresh\big\}$.
\item Return $\widehat{\mkvbdy}_{jk}$
\end{enumerate}
\end{algorithm}

Recall that the forward phase of the IAMB algorithm is the same as the PPS procedure given in Algorithm~\ref{alg:pps}. When Condition~\ref{cond:pps} fails, the estimated Markov boundary after the forward phase is no longer guaranteed to be strictly smaller than the input candidate set without further assumptions. Therefore, the forward phase does not provide a benefit in terms of the sample complexity. Denote 
\[
\entgap_I = \min_j\min_{k\in\layer_{j+1}}\min_{\ell\in\pa(k)}I(X_\ell;X_k\given \anc_j\setminus X_\ell).
\]
Condition~\ref{cond:mi} implies that $\entgap_I>0$.
Then we state the following result when Condition~\ref{cond:pps} is relaxed, recall that $\width=\max_j d_j$ is the width of the DAG.
\begin{theorem}
\label{thm:nopps}
Suppose $\gr$ satisfies the identifiability conditions in Theorem \ref{thm:ident:unequal2:short}. Applying Algorithm~\ref{alg:uneq:short}, and estimate the parents of each node with Algorithm~\ref{alg:iamb}.
If $\unentthresh\le\unentgap/2$, $\ppsthresh\le\entgap_I/2$, and
\[
n\gtrsim \left(\frac{2^d \sqrt{\width \depth / d}}{\min(\ppsthresh, \entgap/2, \unentthresh)\sqrt{\epsilon}}\vee \frac{d^3\width \depth}{(\min(\ppsthresh, \entgap/2, \unentthresh))^2\epsilon} \right),
\]
then $\wh{\gr}=\gr$ with probability $1-\epsilon$. 
\end{theorem}
\noindent 
The proof is deferred to Appendix~\ref{app:nopps}. Theorem~\ref{thm:nopps} shows for a DAG satisfying the identifiability conditions in Theorem~\ref{thm:ident:unequal2:short} but without assuming Condition~\ref{cond:pps}, the current analysis in Appendix~\ref{app:main} of the algorithm has exponential sample complexity in the worst-case. In fact, the sample complexity for recovering layers and then learning parents from layers using Algorithm~\ref{alg:iamb} are similar. The latter simply inflates the former with $d/\depth$. Again, the exponential dependency comes from estimating the entropy.

\subsection{Extension to general distributions}\label{sec:ext:gen}
The proofs of the theorems in this section are analogous to that of Theorem~\ref{thm:main}, and hence omitted.

\paragraph{General discrete distributions}
The sample complexity result in Theorem~\ref{thm:main} can be easily extended to general discrete distributions with finite support size. The proof of Theorem~\ref{thm:ident:unequal2:short} applies without change. 
\begin{theorem}
\label{thm:discrete}
Suppose $\gr$ is an arbitrary DAG with discrete variables, whose support size is bounded by $N$, satisfying the identifiability conditions in Theorem \ref{thm:main}. Algorithm~\ref{alg:uneq:withpps} is applied with minimax entropy estimator and thresholds in Theorem \ref{thm:main}. If $\mbsize\lesssim \log d$ and sample size satisfies
\[
 n \gtrsim \left(\frac{d^2\depth\log^{3} d\log^{2} N}{(\entgap^*_{\unentthresh,\ppsthresh})^2\epsilon} \vee \frac{d^{1+\log N }}{\entgap^*_{\unentthresh,\ppsthresh}\log N}\sqrt{\frac{\depth}{\epsilon\log d}}\right),
\]
then $\wh{\gr}=\gr$ with probability $1-\epsilon$. 
\end{theorem}

\paragraph{Continuous distributions}
Theorem~\ref{thm:ident:unequal2:short} also applies to continuous variables if we replace Shannon entropy with differential entropy. Differential entropy does not preserve all the properties of entropy, e.g. it is not always non-negative and it is not invariant to invertible transformations. Fortunately, differential entropy preserves the essential properties for Theorem~\ref{thm:ident:unequal2:short} to hold. In particular,
since continuous mutual information is still non-negative, the positiveness in Condition~\ref{cond:mi} is still reasonable to assume.

For differential entropy estimation in Algorithm~\ref{alg:uneq:withpps}, we can adopt the minimax estimator from \citet{han2020optimal}, which has the optimal rate over Lipschitz balls. Thus we have sample complexity result in Theorem~\ref{thm:cts}, whose proof is similar with Theorem~\ref{thm:main} in Appendix~\ref{app:main}, simply replacing the estimator for entropy with the one for differential entropy and then applying the result in \citet{han2020optimal}.
\begin{theorem}
\label{thm:cts}
Suppose $\gr$ is an arbitrary DAG with continuous variables, whose densities are over Lipschitz balls with smoothness parameter $0<s\le2$, satisfying the conditions in Theorem~\ref{thm:main}. Algorithm~\ref{alg:uneq:withpps} is applied with differential entropy estimation, and thresholds in Theorem~\ref{thm:main}. If $\mbsize\lesssim \log d$ and sample size satisfies
\[
\left(n\log n \right)^{\frac{2s}{s+\log d}} \vee n \gtrsim \frac{d^2\depth\log d}{(\entgap^*_{\unentthresh,\ppsthresh})^2\epsilon},
\]
then $\wh{\gr}=\gr$ with probability $1-\epsilon$. 
\end{theorem} 

\subsection{Unfaithful example}\label{app:unfaith}
To provide a comparison with the commonly assumed faithfulness assumption, here we construct a simple unfaithful example to illustrate how our approach does not rely on this assumption.

Consider the three-node DAG $Z\to Y$, $Z\to X\to Y$, where $Z$ is a common cause of $X$ and $Y$, and the effect $Z\to Y$ is cancelled by the path $Z\to X\to Y$. Then independence between $Z$ and $Y$ does not imply the $d$-separation between them. Now add one more node $W$ with an edge $W\to X$; See below:
\begin{align*}
    Z\ \ \ \ \ \   & \rightarrow \ \ \ \ Y \\
      \searrow &  \ \ \ \ \nearrow \\
    & X  \\
      \nearrow &  \\
    W \ \ \ \ &
\end{align*}
This BN is still unfaithful due to the independence between $Z$ and $Y$, but can easily be made to satisfy Condition~\ref{cond:ident:main}. To see this, in the first layer, $W$ is used to mask $(X,Y)$ (note that $Z$ is unable to, by independence). Then in the second layer, $X$ will mask $Y$. A concrete example is following:
\begin{align*}
    W & \sim \mathcal{N}(0,\tfrac12) \\
    Z & \sim \mathcal{N}(0,1) \\
    X & = W + \tfrac12 Z + \mathcal{N}(0,\tfrac12) \\
    Y & = X - \tfrac12 Z + \mathcal{N}(0,1)
\end{align*}
We have run experiments to show PC/GES will have SHD around 3 (i.e. very bad for this small model) but the proposed TAM algorithm perfectly recovers the DAG, as expected.

Furthermore, this simple example can be embedded into any DAG satisfying our identifiability condition by treating some sink node as $W$. And we may conjecture that, when unfaithfulness happens in a form of path-cancellation, if we have some other nodes in the ``ancestral'' layer, it is able to help identify the DAG as long as the  ``uncertainty'' relation in \ref{cond:uneq:short} is satisfied.

\section{Proof of Theorem \ref{thm:ident:unequal2:short}}
\label{app:uneq2}

Since Theorem \ref{thm:ident:unequal2:short} is a special case of Theorem~\ref{thm:ident:unequal2} (i.e. $\anprime=\emptyset$), it suffices to prove the latter result.

\begin{proof}[Proof of Theorem \ref{thm:ident:unequal2}]
The proof illustrates the mechanism of Algorithm \ref{alg:uneq} to identify DAGs satisfying conditions in Theorem \ref{thm:ident:unequal2}. If we have identified $\anc_{j}=\cup_{t=0}^{j}\layer_t$, our goal is to distinguish $\layer_{j+1}$ from $V\setminus \anc_{j+1}$. If it is possible, then by induction on $j$ we can complete the proof.

Denote $\ent(X_k \given \anc_{j})=h_{jk}$ for all $k \in V\setminus\anc_{j}$. Sort $h_{jk}$ in ascending order, denote this order by $\tau^{(0)}$, where $\tau^{(0)}$ is a permutation of size $|V\setminus\anc_{j}|$. Any node $X_k \in V\setminus \anc_{j+1}$ must have at least one ancestor in $\layer_{j+1}$. Meanwhile, there exists one important ancestor $X_i\in \an_j(k)$ such that $h_{ji} < h_{jk}$ by \ref{cond:uneq}. 
This implies $\tau^{(0)}_i < \tau^{(0)}_k$ and therefore, $\tau^{(0)}_1$ must be from $\layer_{j+1}$. 

We proceed by considering two operations on the nodes in $V\setminus\anc_{j}$: Testing and Masking. Specifically, we maintain two sequences of set of nodes $\widetilde{L}^{(t)}$ and $\widetilde{S}^{(t)}$, which is indexed by number of testing operations we have conducted (starts with $t=0$ and the same with the superscript of $\tau^{(t)}$). Initialize $\widetilde{L}^{(0)} = \widetilde{S}^{(0)}=\emptyset$, and put nodes being tested or masked into $\widetilde{L}^{(t-1)}$ or $\widetilde{S}^{(t-1)}$ to get $\widetilde{L}^{(t)}$ or $\widetilde{S}^{(t)}$ respectively. 

Then we for each $t=1,2,\cdots$, update 
\begin{align*}
\widetilde{L}^{(t)}=\widetilde{L}^{(t-1)} \cup \{\tau^{(t-1)}_1\}    
\end{align*}
If there are ties, include all the tied nodes into $\widetilde{L}^{(t)}$ and follow the same argument. For $k \in \tau^{(t-1)}_{[2:]}$ compute the mutual information (a.k.a. entropy reduction)
\begin{align*}
    I(X_k; \widetilde{L}^{(t)} \given \anc_{j}) = h_{jk} - \ent(X_k\given \anc_{j}, \widetilde{L}^{(t)} )
\end{align*}
and mask the nodes with positive mutual information, namely update the masked set
\begin{align*}
    \widetilde{S}^{(t)} = \widetilde{S}^{(t-1)}\cup \{k: \mi(X_k;\tilde{L}^{(t)}\given \anc_j) > 0\}.
\end{align*}
Then remove the nodes that have been conditioned or masked from $\tau^{(t-1)}$ to get
\begin{align*}
    \tau^{(t)} = \tau^{(t-1)}\setminus (\widetilde{S}^{(t)}\cup \widetilde{L}^{(t)})
\end{align*}
The following crucial properties of an important ancestor $X_i$ with $\anprime$ of $X_k$ as easy to check: 
\begin{enumerate}
    \item $X_i\indep (\layer_{j+1}\setminus X_i)\given \anc_j$, and thus $\mi(X_i;\layer_{j+1}\setminus X_i\given \anc_j)=0$, so is $\anprime$;
    \item For $X_\ell\in \anprime$, $\tau^{(t)}_\ell < \tau^{(t)}_i$ if $X_i$ and $X_\ell$ are in $\tau^{(t)}$;
    \item $(X_i,\anprime)$ will not be masked and put into $\widetilde{S}^{(t)}$ for all $t$;
    \item If $X_k\in \tau^{(t)}$, then $X_i\in \tau^{(t)}$ and $\tau^{(t)}_i<\tau^{(t)}_k$.
\end{enumerate}
1) is due to $(X_i, \anprime)\subseteq \layer_{j+1}$ and the definition of layer decomposition. 2) is by the definition of $\anprime$, thus $\tau^{(0)}_\ell < \tau^{(0)}_i$. And $\tau^{(t)}$ is a subset of $\tau^{(0)}$, hence the order is preserved. 3) is by 1) otherwise the mutual information is positive. For 4), if $X_i\notin \tau^{(t)}$, then $X_i$ has been included into $\widehat{\layer}^{(t)}$, so is $\anprime$ by 2). Then $X_k$ should have been masked before $t$ since
\begin{align*}
    I(X_k; \widetilde{L}^{(t)} \given \anc_{j}) \ge I(X_k; (X_i,\anprime) \given \anc_{j}) > 0.
\end{align*}
$\tau^{(t)}_i<\tau^{(t)}_k$ is due to the same reason as 2).

Then we conclude that $\tau^{(t)}_1$ must be from $\layer_{j+1}$ by 4).
By continuing this procedure, all the nodes are eventually tested or masked after say $t^*$ steps, i.e., $V\setminus \anc_j = \widetilde{L}^{(t^*)}\cup \widetilde{S}^{(t^*)}$.
Since nodes in $\widetilde{L}^{(t^*)}$ are composed of $\tau_1^{(t)}$ for $t=1,2,\cdots$, thus $\widetilde{L}^{(t^*)}\subseteq \layer_{j+1}$. Then we further claim for any $\ell\in \widetilde{S}^{(t^*)}$, $X_\ell\notin \layer_{j+1}$. Suppose $X_\ell$ is included at step $t$, otherwise we will have
\begin{align*}
    0 < \mi(X_\ell; \widetilde{L}^{(t)}\given \anc_j) \le \mi(X_\ell; \layer_{j+1}\given \anc_j) = 0
\end{align*}
The last equality is due to the mutual independence of nodes in $\layer_{j+1}$ given $\anc_j$. Therefore, the final $\widetilde{L}^{(t^*)}$ is exactly $\layer_{j+1}$.
\end{proof}

\section{Proof of Theorem \ref{thm:main}}\label{app:main}
We prove the theorem in two steps. After establishing some preliminary bounds, we prove Proposition~\ref{prop:pps}, which gives the sample complexity for using PPS to recover Markov boundaries and estimate the  conditional entropy under Condition~\ref{cond:pps}. Then we establish the sample complexity of the layer-wise learning framework for DAGs satisfying identifiability conditions in Theorem \ref{thm:ident:unequal2:short}.

\subsection{Preliminary bounds}\label{app:main:prebound}

To derive the sample complexity of Algorithm~\ref{alg:uneq:withpps}, we borrow the minimax estimator of entropy from \citet{wu2016minimax}. Denote this estimator by $\wh{\ent}(X)$.
\begin{lemma}
\label{lem:entest}
For discrete random variables $X_1,\cdots,X_d$ with distribution $P_1,\cdots,P_d$ and support size $K_{1},\cdots,K_{d}$, let $K=\max_k K_{k}$. If $n\gtrsim \frac{K}{\log K}$, then 
\[
\sup_{k}\truepr\left[\left|\ent(X_k)-\wh{\ent}(X_k)\right|\ge t \right]\lesssim \left[  \left(\frac{K}{n\log K}\right)^2 + \frac{\log^2 K}{n} \right]/t^2
\]
\end{lemma}
\begin{proof}
First we pad all the random variables such that they have the same support size $K$ by assigning the extra support with zero probability. Then for all $X_k$
\[
\truepr\left[\left|\ent(X_k)-\wh{\ent}(X_k)\right| \ge t\right] \le \frac{\E(\ent(X_k)-\wh{\ent}(X_k))^2}{t^2}
\]
Thus, by Proposition 4 in \citet{wu2016minimax},
\begin{align*}
\sup_{k}\truepr\left[\left|\ent(X_k)-\wh{\ent}(X_k)\right| \ge t\right] 
&\le \frac{\sup_{k}\E(\ent(X_k)-\wh{\ent}(X_k))^2}{t^2} \\
&\lesssim \frac1{t^{2}}\left[\left(\frac{K}{n\log K}\right)^2 + \frac{\log^2 K}{n} \right].
\end{align*}
\end{proof}

We next establish some preliminary uniform bounds on the estimation error of conditional entropy and mutual information using Lemma~\ref{lem:entest}.
Suppose we estimate the conditional entropy $\ent(X_k\given \anc)$ for any $k\in V$ and its some ancestral set $\anc$ by
\begin{align}\label{eq:entest}
\wh{\ent}(X_k\given \anc) = \wh{\ent}(X_k,\anc) - \wh{\ent}(\anc).
\end{align}
Then because we are dealing with binary variables, by Lemma~\ref{lem:entest}
\begin{align}
\nonumber
  \sup_{k, \anc}&\truepr\left[\left|\widehat{\ent}(X_k\given \anc)-\ent(X_k\given \anc)\right|\ge t\right] \\
\nonumber
   &\le \sup_{k,\anc} \frac1{t^{2}}\left(\widehat{\ent}(X_k\given \anc)-\ent(X_k\given \anc)\right)^2 \\
\nonumber
   &\le \sup_{k,\anc} \frac1{t^{2}}\left\{ 2\left(\widehat{\ent}(X_k,\anc)-\ent(X_k,\anc)\right)^2+2\left(\widehat{\ent}(\anc)-\ent(\anc)\right)^2  \right\}\\
\label{eq:condent:bound}
   &\les \frac{\delta_{|\anc|+ 1}^2}{t^2}
\end{align}
where 
\begin{align}
\label{eq:def:delta}
\delta_{p}^2 \asymp \left[\left(\frac{2^{p}}{n p}\right)^2 + \frac{p^2}{n}\right].
\end{align}
Note that $\delta^2_{p}$ is an increasing function of $p$. When $\anc=\emptyset$, and the $|\anc|=0$. Thus $\delta^2_{|\anc|+1}$ reduces to estimating entropy of a binary random variable, which is of parametric rate $1/n$.
Similarly, if we try to estimate the conditional mutual information by using sample version of the identity
\begin{align*}
    \mi(X_k;X_\ell\given \anc) &= \mi(X_k;(X_\ell,\anc)) - \mi(X_k;\anc) \\
    &= \ent(X_k) + \ent(X_\ell,\anc) - \ent(X_k,X_\ell,\anc) \\
    &- \ent(X_k) - \ent(\anc) + \ent(X_k,\anc) \\
    &= \ent(X_\ell,\anc) - \ent(X_k,X_\ell,\anc)- \ent(\anc) + \ent(X_k,\anc)
\end{align*}
The estimation error is again dominated by the second term, which has the largest support size:
\begin{align}\label{eq:mi:bound}
    \sup_{k,\ell,\anc}\truepr \left[\left|\widehat{\mi}(X_k;X_\ell\given \anc)-\mi(X_k;X_\ell\given \anc)\right|\ge t\right] \les \frac{\delta_{|\anc|+ 2}^2}{t^2}
\end{align}
The factor of $+2$ is not important when size $|\anc|$ is large, thus for simplicity we absorb it into the constant before them. We will present the estimation error bound for $\ent(X_\ell\given \anc)$ and $\mi(X_k;X_\ell\given \anc)$ by $C\delta^2_{|\anc|}/t^2$ for some constant $C$.

\subsection{Proof of Proposition~\ref{prop:pps}}\label{app:samcom:pps}
In this section, we prove the Proposition~\ref{prop:pps}. Recall that $\mbsize_{\anc k} := |\MB(k;\anc)|$.
\begin{proof}[Proof of Proposition~\ref{prop:pps}]
For any node $k\in V$ and its ancestor set $\anc$, and node $\ell\in \anc$, \eqref{eq:mi:bound} implies
\[
\sup_{\anc'\subseteq \anc, |\anc'|\le \mbsize_{\anc k}}\truepr\left[\left|\widehat{\mi}(X_k;X_\ell\given \anc')-\mi(X_k; X_\ell\given \anc')\right|\ge t\right]
\les \frac{\delta_{\mbsize_{\anc k}}^2}{t^2}.
\]
For the first step, with probability $1-|\anc| \delta^2_{\mbsize_{\anc k}}/t^2$ we have for all $X_\ell\in \anc$
\[
\left|\widehat{I}(X_k; X_{\ell})-I(X_k; X_\ell)\right| < t
\]
Thus for all $X_{\ell'}\in \anc\setminus\mkvbdy_{\anc k}$,
\[
\widehat{I}(X_k; X_{\ell^*})-\widehat{I}(X_k; X_{\ell'}) > \tilde{\entgap}_{\anc k}-2t
\]
where $\ell^*=\argmax_{i:X_i\in \mkvbdy_{\anc k}} I(X_k;X_i)$. So we only need $t< \tilde{\entgap}_{\anc k}/2$ to ensure we include a node in $\mkvbdy_{\anc k}$ rather than in $\anc\setminus\mkvbdy_{\anc k}$. Following the same argument, when we have found a proper subset $m\subsetneq \mkvbdy_{\anc k}$, with probability $1-\left(|\anc|-|m|\right) \delta^2_{\mbsize_{\anc k}}/t^2$ we have for all $X_\ell\in \anc\setminus m$
\[
\left|\widehat{I}(X_k; X_{\ell}\given m)-I(X_k; X_\ell\given m)\right| < t.
\]
Thus for all $X_{\ell'}\in \anc\setminus\mkvbdy_{\anc k}$
\[
\widehat{I}(X_k; X_{\ell^*}\given m)-\widehat{I}(X_k; X_{\ell'}\given m) > \tilde{\entgap}_{\anc k}-2t
\]
where $\ell^*=\argmax_{i:X_i\in \mkvbdy_{\anc k}\setminus m} I(X_k;X_i\given m)$. So we only need $t\le \tilde{\entgap}_{\anc k}/2$ to ensure we do not include any nodes from $\anc\setminus\mkvbdy_{\anc k}$. At the same time, 
\[
\widehat{I}(X_k; X_{\ell^*}\given m)> I(X_k; X_{\ell^*}\given m)- t \ge 2\ppsgap_{\anc k}-t
\]
To avoid triggering the threshold, we need 
\[
\widehat{I}(X_k; X_{\ell^*}\given m)> 2\ppsgap_{\anc k}-t \ge \ppsthresh.
\]
So $t \le \ppsthresh$ will do the job. After at most $\mbsize_{\anc k}$ steps, we have recovered $\mkvbdy_{\anc k}$, then requiring $t \le \ppsthresh$ will trigger the stopping criterion. Since for any $X_{\ell'}\in \anc\setminus\mkvbdy_{\anc k}$,
\[
\left|\widehat{I}(X_k; X_{\ell'}\given \mkvbdy_{\anc k})-I(X_k; X_{\ell'}\given \mkvbdy_{\anc k})\right|=\widehat{I}(X_k; X_{\ell'}\given \mkvbdy_{\anc k}) < t \le \ppsthresh.
\]
Thus in conclusion, we can recover $\mkvbdy_{\anc k}$ for $X_k$ with probability
\begin{align*}
    \truepr(\wh{\mkvbdy}_{\anc k}=\mkvbdy_{\anc k})&\ge \prod_{i=0}^{\mbsize_{\anc k}}\left(1- \left(|\anc|-|i|\right) \frac{\delta^2_{\mbsize_{\anc k}}}{t^2}\right) \\
    & \ge 1 - \left((\mbsize_{\anc k}+1)|\anc| -\frac{\mbsize_{\anc k}(\mbsize_{\anc k}+1)}{2}\right)\frac{\delta^2_{\mbsize_{\anc k}}}{t^2} \\
    & \ge 1 - (\mbsize_{\anc k}+1)|\anc| \frac{\delta^2_{\mbsize_{\anc k}}}{t^2}.
\end{align*}
Furthermore, by combining this with \eqref{eq:condent:bound}, we can estimate the conditional entropy as
\begin{align*}
    & \truepr\left[\left|\widehat{\ent}(X_k\given \anc)-\ent(X_k\given \anc)\right|<t', \wh{\mkvbdy}_{\anc k}=\mkvbdy_{\anc k}\right]\\
    =& \truepr(\wh{\mkvbdy}_{\anc k}=\mkvbdy_{\anc k})\truepr\left[\left|\widehat{\ent}(X_k\given \anc)-\ent(X_k\given \anc)\right|<t' \given \wh{\mkvbdy}_{\anc k}=\mkvbdy_{\anc k}\right]\\
    =& \truepr(\wh{\mkvbdy}_{\anc k}=\mkvbdy_{\anc k})\truepr\left[\left|\widehat{\ent}(X_k\given \mkvbdy_{\anc k})-\ent(X_k\given \mkvbdy_{\anc k})\right|<t'\right]\\
    \ge& \Big(1 - (\mbsize_{\anc k}+1)|\anc| \frac{\delta^2_{\mbsize_{\anc k}}}{t^2}\Big) \Big(1-\frac{\delta^2_{\mbsize_{\anc k}}}{{t'}^2}\Big)\\
    \ge& \Big(1 - (\mbsize_{\anc k}+1)|\anc| \frac{\delta^2_{\mbsize_{\anc k}}}{t^2}-\frac{\delta^2_{\mbsize_{\anc k}}}{{t'}^2}\Big)
\end{align*}
Let $t'=t \le \min(\ppsthresh, \tilde{\entgap}_{\anc k}/2)$,
then we have
\[
\truepr\left[\left|\widehat{\ent}(X_k\given \anc)-\ent(X_k\given \anc)\right|<t, \wh{\mkvbdy}_{\anc k}=\mkvbdy_{\anc k}\right]\ge 1- (\mbsize_{\anc k}+2)|\anc| \frac{\delta^2_{\mbsize_{\anc k}}}{t^2}
\]
The argument above holds for all $X_k\in V$ and its ancestor set $\anc$, which completes the proof.
\end{proof}

\subsection{Proof of Theorem~\ref{thm:main}}\label{app:samcom:main}
Now we are ready to prove the main theorem on the sample complexity of Algorithm~\ref{alg:uneq:withpps}.
\medskip
\begin{proof}[Proof of Theorem~\ref{thm:main}]

Define events by $\mathcal{E}_0=\emptyset$ and
\[
\mathcal{E}_{j}=\left\{\wh{\layer}_j=\layer_j \AND \wh{\mkvbdy}_{(j-1) k}=\mkvbdy_{(j-1) k} \text{ for } k\in\layer_j\right\}
\]
for $j=1,\ldots,r$. Then
 \[
 \truepr(\wh{\gr}=\gr) = \prod_{j=1}^{\depth}\truepr(\mathcal{E}_{j} \given \mathcal{E}_{j-1}).
 \]
 For the first step, with probability $1-d\delta_{0}^2/t^2$ we have for all $k=1,\cdots,d$
\[
|\ent(X_k)-\widehat{\ent}(X_k)|< t 
\]
which implies for $k\notin \layer_1$ and its corresponding $X_i$
\begin{align*}
\widehat{\ent}(X_k)  - \widehat{\ent}(X_i) >  \Delta - 2t
\end{align*}
With $t \le \Delta / 2$, we have $X_i$ comes before $X_k$ in the order $\widehat{\tau}$ of estimated marginal entropies. Conducting the TAM step, for each $X_i$ comes forst in the ordering $\widehat{\tau}$, we estimate the mutual information $\mi(X_i;X_k)$ for all remaining nodes $X_k$ in $\widehat{\tau}$. With probability at least $1 - d\delta^2_2/ t^2$, we have for all remaining $k$
\begin{align*}
    |\mi(X_i;X_k) - \widehat{\mi}(X_i;X_k)| < t
\end{align*}
Therefore,
\begin{align*}
    \begin{cases}
    \mi(X_i;X_k) < 2t & X_i \text{ is not the ancestor of }X_k \\
    \mi(X_i;X_k) > \unentgap - 2t & X_i \text{ is the ancestor of }X_k
    \end{cases}
\end{align*}
With $\unentthresh:=t\le \unentgap/2$, $X_k$ would be masked. There are at most $d_1:=|\layer_1|$ many $X_i$'s thus there is at most $d_1\times d$ mutual information need to be estimated correctly. So with probability at least $1-d_1d\delta^2_2/t^2$, the TAM step succeeds recovering $\layer_1$. And then
\[
\truepr(\mathcal{E}_1) \ge 1 - \frac{d\delta^2_0}{\entgap^2/4} - \frac{d_1d\delta^2_2}{\unentthresh^2}
\]
Following the same argument, after $j$ loops, given the layers in $\anc_j$ are correctly identified, we invoke Proposition~\ref{prop:pps} with $\anc=\anc_j$ to have for all $k\in V\setminus\anc_{j}$, with probability at least
\[
1-\Big(\sum_{s=j+1}^{\depth}d_s\Big) \Big(2\mbsize_j\sum_{s=1}^jd_s \frac{\delta_{\mbsize_j}^2}{t^2}\Big),
\]
using PPS procedure to estimate the conditional entropeis and Markov boundaries gives
\[
|\widehat{\ent}(X_k\given \anc_{j})-\ent(X_k\given \anc_{j})|< t\ \  \AND \ \ \wh{\mkvbdy}_{jk} = \mkvbdy_{jk}.
\]
which implies for $k\notin \layer_{j+1}$ and its corresponding $X_i$,
\begin{align*}
    \widehat{\ent}(X_k\given \anc_j) - \widehat{\ent}(X_i\given \anc_j) > \Delta - 2t
\end{align*}
So $t\le \Delta/2$ amounts to $X_i$ coming before $X_k$ in the order $\widehat{\tau}$ of estimated conditional entropies of remaining nodes. Note that $t$ also needs to satisfy $t\le \min_k(\ppsthresh, \widetilde{\entgap}_{jk}/2)$. Conducting TAM step, estimate the conditional mutual information using the identity
\begin{align*}
    \mi(X_k;X_i\given \anc_j) &= \ent(X_k\given \anc_j) - \ent(X_k\given X_i,\anc_j)\\
    &= \ent(X_k\given \mkvbdy_{jk}) - \ent(X_k\given X_i,\mkvbdy_{jk})
\end{align*}
Invoking \eqref{eq:condent:bound}, since $\mkvbdy_{jk}$ is already identified, with probability at least $1 - d\delta^2_{\mbsize_j}/ t^2$, we have for all remaining $k$
\begin{align*}
    |\mi(X_i;X_k) - \widehat{\mi}(X_i;X_k)| < t
\end{align*}
Therefore,
\begin{align*}
    \begin{cases}
    \mi(X_i;X_k) < 2t & X_i \text{ is not the ancestor of }X_k \\
    \mi(X_i;X_k) > \unentgap - 2t & X_i \text{ is the ancestor of }X_k
    \end{cases}
\end{align*}
With $\unentthresh :=t\le \unentgap/2$, $X_k$ would be masked while other nodes remains unmasked. There are at most $d_{j+1}:=|\layer_{j+1}|$ manny $X_i$'s thus there is at most $d_{j+1}\times d$ conditional mutual information need to be estimated correctly. So with probability at least $1-d_{j+1}d\delta^2_{\mbsize_j}/t^2$, the TAM step succeeds recovering $\layer_{j+1}$. Combine the PPS step and TAM step together, 
\[
\truepr(\mathcal{E}_{j+1}\given \mathcal{E}_j) \ge 1 - \frac{d_{j+1}d\delta^2_{\mbsize_j}}{\unentthresh^2} - \frac{2\mbsize_j d^2\delta^2_{\mbsize_j}}{\bigg(\min_k(\entgap/2,\ppsthresh, \widetilde{\entgap}_{jk}/2)\bigg)^2}
\]
In conclusion, we have
\begin{align*}
\truepr(\wh{\gr}=\gr) 
&= \prod_{j=0}^{\depth-1}\truepr(\mathcal{E}_{j+1}\given \mathcal{E}_j) \\
&\ge \left(1 - \frac{d\delta^2_0}{\entgap^2/4} - \frac{d_1d\delta^2_2}{\unentthresh^2} \right) \prod_{j=1}^{\depth-1}\left( 1 -  \frac{d_{j+1}d\delta^2_{\mbsize_j}}{\unentthresh^2} - \frac{2\mbsize_j d^2\delta^2_{\mbsize_j}}{\bigg(\min_k(\entgap/2,\ppsthresh, \widetilde{\entgap}_{jk}/2)\bigg)^2} \right)\\
&\ge  1 - \left(1 - \frac{d\delta^2_0}{\entgap^2/4} - \frac{d_1d\delta^2_2}{\unentthresh^2}\right) - \sum_{j=1}^{\depth-1}\left(  \frac{d_{j+1}d\delta^2_{\mbsize_j}}{\unentthresh^2} - \frac{2\mbsize_j d^2\delta^2_{\mbsize_j}}{\bigg(\min_k(\entgap/2,\ppsthresh, \widetilde{\entgap}_{jk}/2)\bigg)^2}
\right) \\
&\ge 1 - \frac{4\mbsize \depth d^2}{\bigg(\min_{jk}(\entgap/2,\unentthresh,\ppsthresh, \widetilde{\entgap}_{jk}/2)\bigg)^2}\delta^2_\mbsize
\end{align*}
where $\unentthresh \le \unentgap/2$ and $\ppsthresh \le \min_{jk}\ppsgap_{jk}$. Since $\mbsize\lesssim \log d$, $\delta_\mbsize^2$ (cf. \eqref{eq:def:delta}) is dominated by its second term. Requiring $\truepr(\wh{\gr}=\gr) > 1-\epsilon$, we have the final result
\begin{align*}
n &\gtrsim \frac{d^2\depth\log^{3} d}{\bigg(\min_{jk}(\entgap/2,\unentthresh,\ppsthresh, \widetilde{\entgap}_{jk}/2)\bigg)^2\epsilon}.
\end{align*}
This completes the proof.

\end{proof}

\subsection{Tuning parameters}
\label{app:tuning}

In practice, the quantities $\unentgap,\ppsgap_{jk}$ needed in Theorem~\ref{thm:main} may not be known. Theorem~\ref{thm:tuning} below remedies this by prescribing data-dependent choices for $\unentthresh$ and $\ppsthresh$:
\begin{theorem}\label{thm:tuning}
Suppose the conditions in Theorem~\ref{thm:main} are satisfied with the stronger sample size requirement
\begin{align*}
n\gtrsim \frac{d^3\log^{3} d}{\epsilon^2 \wedge (\entgap^*_{\unentgap/2,\ppsgap_{jk}})^4},
\end{align*}
where $\entgap^*_{\unentgap/2,\ppsgap_{jk}}$ is defined in Theorem~\ref{thm:main} with $\unentgap/2$ and $\ppsgap_{jk}$ plugged in. By choosing %
\begin{align*}
\unentthresh = \ppsthresh \asymp (d^3\log d)^{1/4}\Big[\Big(\frac{d}{n\log d}\Big)^2 +\frac{\log d}{\sqrt{n}}\Big]^{1/4}
\end{align*}
in Algorithm~\ref{alg:uneq:withpps}, we have $\widehat{\gr}=\gr$ with probability $1-\epsilon$.
\end{theorem}

\begin{proof}
In the proof of Proposition~\ref{prop:pps}, for $X_k\in\layer_{j+1}$, we need $t$ to be small enough when doing estimation such that
\begin{align*}
    0 & \le \tilde{\entgap}_{jk} - 2t
    \quad\AND\quad
    t \le \ppsthresh \le 2\ppsgap_{jk} - t.
\end{align*}
Similarly, in the proof of Theorem~\ref{thm:main} we need 
\[
\unentthresh = t \le \unentgap/2.
\]
Though the $t$'s here are different, we can take the minimum of them, then it suffices to take
\[
\unentthresh=\ppsthresh = t=  (d^3\mbsize)^{1/4}\delta^{1/2}_{\mbsize}
\]
and require
\[
(d^3\mbsize)^{1/4}\delta^{1/2}_{\mbsize}\le \min_{jk}(\entgap/2,\unentgap/2,\ppsgap_{jk}, \tilde{\entgap}_{jk}/2).
\]
Then we have 
\begin{align*}
    \truepr(\wh{\gr}=\gr)&\ge 1 - \frac{4d^2\depth \mbsize \delta_\mbsize^2}{\bigg(\min_{jk}(\entgap/2,\unentthresh,\ppsthresh, \widetilde{\entgap}_{jk}/2)\bigg)^2}\\
    & \ge 1 - \frac{4d^3 \mbsize \delta_\mbsize^2}{\bigg(\min_{jk}(\entgap/2,\unentthresh,\ppsthresh, \widetilde{\entgap}_{jk}/2)\bigg)^2}\\
    &\ge 1 - 4d\sqrt{d\mbsize}\delta_{\mbsize}.
\end{align*}
Plugging in the exact form of $\delta_\mbsize$ with $\mbsize\lesssim \log d$, the final sample complexity is as desired.
\end{proof}

\section{Proof of Theorem \ref{thm:poly}}\label{app:poly}
First, we need a lemma similar to Lemma~\ref{lem:mi}:
\begin{lemma}\label{lem:mi:mb}
 $I(\rv_{k};\mkvbdy_{jk}\given \anc)>0$ for any subset $\anc\subset\anc_j$ such that $\mkvbdy_{jk}\setminus\anc\ne\emptyset$.
\end{lemma}
The proof follows the one for Lemma~\ref{lem:mi} in Appendix~\ref{app:mi}. To see this, we can simply replace $\pa(k)$ in the arguments with $\mkvbdy_{jk}$, and replace the minimality of $\gr$ with minimality of $\mkvbdy_{jk}$. We will also need the following lemma:
\begin{lemma}
For any $X_c\in \mkvbdy_{jk}$, there is at least one directed path from $X_c$ to $X_k$.
\end{lemma}
\begin{proof}
If there is no path between $X_c$ and $X_k$, $X_c\indep X_k$, $X_c$ cannot be in $\mkvbdy_{jk}$. If $X_c$ and $X_k$ are only connected by undirected paths, $X_c$ would not be in ancestors or descendants of $X_k$. If some paths are through some descendants (children) of $X_k$, the descendants would serve as colliders to block the paths, so the effective paths must be through the ancestors of $X_k$.

For any undirected path connects them through some ancestor, say $X_a$, if the edge from $X_c$ is not pointing to $X_a$, $X_a$ would serve as a common cause. If it is pointing to $X_a$, we move $X_a$ one node toward $X_c$, until there is a common cause, otherwise $X_c$ is connected to $X_k$ with a directed path. When we find the common cause, if $X_a$ is an ancestor of $X_c$, then it is in $\anc_j$ and block the path if conditioned on. If $X_a$ is not an ancestor of $X_c$, then there must be a change of edge direction, so there exits a collider on the path between $X_a$ and $X_c$. If the collider is not in $\anc_j$, it will block the path; if it is in $\anc_j$, $X_a$ will also be in $\anc_j$, so the path is blocked when $\anc_j$ is conditioned on. As a result of all the cases above,
\[
\truepr(X_k\given \anc_j) = \truepr(X_k\given \anc_j\setminus X_c)  
\]
so $X_c$ cannot be in $\mkvbdy_{jk}$.
\end{proof}

\noindent
Finally we prove the theorem:
\medskip
\begin{proof}[Proof of Theorem \ref{thm:poly}]
Condition~\ref{cond:pps} is constructive for PPS procedure. Hence we prove it by showing that nodes not belong to the desired Markov boundary will not be chosen when conducting PPS on poly-forest. Suppose we have identified $j$-th layer, for any $X_k\in V\setminus\anc_{j}$, we try to find $\mkvbdy_{jk}$. For any node $X_\ell \in \anc_j\setminus \mkvbdy_{jk}$ we should not consider, if $X_\ell$ is disconnected to $X_k$, which means there is no path connecting them together, then $X_k\indep X_\ell\given m$, $I(X_k;X_\ell\given m)=0$ for any $m\subsetneq \mkvbdy_{jk}$, we do not need to worry about it. 

If $X_\ell$ is connected to $X_k$ through a child of $X_k$, then this must be an undirected path. Since if this is a directed path, it will be from $X_k$ to $X_\ell$, which is contradicted to $X_\ell \in \anc_j$. Moreover for this undirected path, since the edge is from $X_k$ to its child at first, there will be a change of direction, which serves as a collider blocking this and the only path connecting $X_k$ and $X_\ell$, thereafter $X_k$ and $X_\ell$ are $d$-separated by any subset $m\subsetneq \mkvbdy_{jk}$ and $X_k\indep X_\ell\given m$, $I(X_k;X_\ell\given m)=0$, we do not need to worry about it.

If $X_\ell$ is connected to $X_k$ through one parent $X_t$ or $X_\ell$ is the parent, we can divide it into two situations, whether there is $X_c\in \mkvbdy_{jk}$ on the path or not. If there is one $X_c$ on the path, since $X_c$ has exactly one directed path leading to $X_k$, this will be part of the path connecting $X_\ell$ and $X_k$. For two edges of $X_c$ on this path, $X_c$ has an edge out on the path toward $X_k$. For the other edge, no matter the direction is, either in or out, $X_c$ will block the path if conditioned on and $d$-separates $X_k$ and $X_\ell$. Then we have for any $m\subsetneq \mkvbdy_{jk}$
\[
X_k\indep X_{\ell}\given X_c,m
\]
Using this property and decomposition
\begin{align*}
    I(X_k;X_c,X_{\ell}\given m)&=I(X_k;X_c\given X_{\ell}, m) + I(X_k;X_{\ell}\given m)\\
     &= I(X_k;X_{\ell}\given X_c,m) + I(X_k;X_c\given m)=I(X_k;X_c\given m)
\end{align*}
So 
\[
I(X_k;X_c\given m)=I(X_k;X_c\given X_{\ell},m) + I(X_k;X_{\ell}\given m) > I(X_k;X_{\ell}\given m)
\]
For the last inequality, we trigger Lemma~\ref{lem:mi:mb}. So we will not select $X_{\ell}$ at any step of PPS procedure. 

If there is no $X_c$ on the path, then for any $X_c\in \mkvbdy_{jk}$, since $X_c$ is connected to $X_k$ through a directed path, and $X_\ell$ is connected to $X_k$ through one parent $X_t$ or $X_\ell$ is the parent, then $X_c$ and $X_\ell$ is connected by and only by the combination of two paths $X_c\rightarrow \cdots \rightarrow X_k$ and $X_\ell -\cdots - X_t\rightarrow X_k$. If these two paths converge at $X_k$, $X_k$ will serve as a collider to block this path. If they converge at some node before $X_k$, denoted as $X_{u}$, which is not in $\mkvbdy_{jk}$. $X_u$ cannot be $X_\ell$ otherwise it will $d$-separate $X_k$ and $X_c$ such that $X_c\notin \mkvbdy_{jk}$. Since the path $X_c\rightarrow \cdots \rightarrow X_k$ is directed, the edge on this path from $X_c$ is pointing to $X_u$. If the edge from $X_\ell$ is also pointing to $X_u$, $X_u$ will serve as a collider to block this path between $X_\ell$ and $X_c$. So $X_c$ and $X_\ell$ are $d$-separated by empty set in these two cases, furthermore $\mkvbdy_{jk}\indep X_\ell$. Since $X_\ell\indep X_k \given \mkvbdy_{jk}$. Therefore we have 
\begin{align*}
    \truepr(X_\ell, X_k,\mkvbdy_{jk}) &= \truepr(X_\ell,X_k\given \mkvbdy_{jk})\truepr(\mkvbdy_{jk}) \\
    &= \truepr(X_\ell \given \mkvbdy_{jk})\truepr(X_k\given \mkvbdy_{jk})\truepr(\mkvbdy_{jk}) \\
    &= \truepr(X_\ell) \truepr(X_k, \mkvbdy_{jk})
\end{align*}
Hence $X_\ell \indep (X_k,\mkvbdy_{jk})$ and for any $m\subsetneq \mkvbdy_{jk}$
\begin{align*}
    0 &= I(X_\ell; (X_k,\mkvbdy_{jk}))\\
    &= I(X_\ell;X_k\given m) + I(X_\ell; m) + I(X_\ell; \mkvbdy_{jk}\setminus m\given X_k,m)
\end{align*}
By non-negativity of conditional mutual information, we have $I(X_\ell;X_k\given m)=0$. So we don not need to worry about it either.

If the edge from $X_\ell$ is not pointing to $X_u$ but on the opposite, $X_\ell$ and $X_c$ may not be independent. However, for any other node $X_{c'}\in \mkvbdy_{jk}\setminus X_c$, if $X_\ell$ is connected with $X_{c'}$ through $X_k$, $X_k$ serves as a collider to block the path. If $X_\ell$ is connected with $X_{c'}$ through $X_c$, conditioning on $X_c$ will block the path. Therefore, $X_c$ $d$-separates $X_\ell$ and $X_{c'}$ so $X_\ell \indep (\mkvbdy_{jk}\setminus X_c) \given X_c$. As a result,
\begin{align*}
\truepr(X_\ell, \mkvbdy_{jk}\setminus X_c\given X_c) =& \truepr(X_\ell \given X_c)\truepr(\mkvbdy_{jk}\setminus X_c\given X_c)\\
\truepr(X_\ell \given X_c) =& \truepr(X_\ell \given \mkvbdy_{jk})
\end{align*}
Thus
\begin{align*}
    \truepr(X_\ell, X_k, \mkvbdy_{jk}) = & \truepr(X_\ell\given \mkvbdy_{jk})\truepr(X_k\given \mkvbdy_{jk})\truepr(\mkvbdy_{jk})\\
    =& \truepr(X_\ell \given X_c)\truepr(X_k, \mkvbdy_{jk})
\end{align*}
Meanwhile,
\[
     \truepr(X_\ell, X_k, \mkvbdy_{jk}) = \truepr(X_\ell\given X_k,\mkvbdy_{jk})\truepr(X_k,\mkvbdy_{jk})
\]
Hence
\[
\truepr(X_\ell\given X_k,\mkvbdy_{jk}) = \truepr(X_\ell \given X_c)\ \ \ \ X_\ell \indep (X_k,\mkvbdy_{jk}\setminus X_c) \given X_c\ \ \ \ I(X_\ell; (X_k,\mkvbdy_{jk}\setminus X_c) \given X_c)=0
\]
By the same decomposition of mutual information, we can have for any subset $m$, $I(X_\ell;X_k \given X_c, m)=0$, thereafter we will not select $X_{\ell}$ at any step of PPS procedure. 
\end{proof}

\section{Condition \ref{cond:uneq:mi:short} and poly-forests}\label{app:lem:poly}

Condition \ref{cond:uneq:mi:short} is a nondegeneracy condition on the distribution $P$ that may be violated, for example, when there is path cancellation. In this appendix, we show that a sufficient (but not necessary) condition for Condition \ref{cond:uneq:mi:short} is that there exists at most one directed path between any two nodes in the graph, in other words, when path cancellation is impossible.

\begin{lemma}\label{lem:poly}
If $\gr$ is a poly-forest, then Condition \ref{cond:uneq:mi:short} holds.
\end{lemma}

\begin{proof}
Let's drop the subscript and use the notation $X,Z$ for $X_i,X_k$. For poly-forest, there exists and only exists one direct path from ancestor $X$ to descendant $Z$. With loss of generality, let $\anc_j=\emptyset$, since the Markov property of the subgraph $G[V\setminus \anc_j]$ does not change. If $X\in\pa(Z)$, by minimality, all edges are effective, thus $I(X;Z) > 0$.

If the directed path is formed by three nodes $X\to Y \to Z$. Suppose the contrary, $X\indep Z$, we have following:
\begin{align*}
    P(X)P(Z) &= P(X,Z) \\
    & = \sum_{y\in\{0,1\}}P(X,Y=y,Z) \\
    & = P(X)\sum_{y\in\{0,1\}}P(Z\given Y=y)P(Y=y\given X)
\end{align*}
The first equality is by independence, the third one is by Markov property. Then by positivety of the probability, we have
\begin{align*}
    \sum_{y\in\{0,1\}}P(Z\given Y=y)P(Y=y\given X) &= P(Z) \\
    &= \sum_{y\in\{0,1\}}P(Z\given Y=y)P(Y=y)
\end{align*}
After rearrangement,
\[
\sum_{y\in\{0,1\}}P(Z\given Y=y)\bigg[P(Y=y\given X) - P(Y=y)\bigg]=0
\]
Since $X\not \indep Y$, thus $P(Y=y\given X)\ne P(Y=y)$, therefore,
\[
\frac{P(Z\given Y=1)}{P(Z\given Y=0)} = -\frac{P(Y=0\given X) - P(Y=0)}{P(Y=1\given X) - P(Y=1)} = -\frac{1 - P(Y=1\given X) - 1  + P(Y=1)}{P(Y=1\given X) - P(Y=1)} = 1
\]
Thus $P(Z\given Y=1)=P(Z\given Y=0)$ contradicts that $Y\not\indep Z$.

More generally, let the directed path be formed by $X\to Y_1 \to \cdots \to Y_p \to Z$. Suppose $X\indep Z$, we have:
\begin{align*}
    P(X)P(Z) &= P(X,Z) \\
    & = \sum_{y_1\in\{0,1\}}\cdots\sum_{y_d\in\{0,1\}}P(X,Y_1=y_1,\cdots,Y_d=y_d,Z) \\
    & = P(X)\sum_{y_d\in\{0,1\}}\sum_{y_1\in\{0,1\}}P(Z\given Y_d=y    _d)P(Y_d=y_d\given Y_1=y_1)P(Y_1=y_1\given X)
\end{align*}
After rearrangement, 
\[
\frac{P(Z\given Y_d=1)}{P(Z\given Y_d=0)} = - \frac{\sum_{y_1\in \{0,1\} } [P(Y_1=y_1\given X) - P(Y_1=y_1)] P(Y_d=1\given Y_1=y_1) }{\sum_{y_1\in \{0,1\} } [P(Y_1=y_1\given X) - P(Y_1=y_1)] P(Y_d=0\given Y_1=y_1)}
\]
Look at the numerator
\begin{align*}
    &\sum_{y_1\in \{0,1\} } \bigg[P(Y_1=y_1\given X) - P(Y_1=y_1)\bigg] P(Y_d=1\given Y_1=y_1) \\
    =& \sum_{y_1\in \{0,1\} } \bigg[P(Y_1=y_1\given X) - P(Y_1=y_1)\bigg] \bigg[1 - P(Y_d=0\given Y_1=y_1)\bigg] \\
    =& - \sum_{y_1\in \{0,1\} } \bigg[P(Y_1=y_1\given X) - P(Y_1=y_1)\bigg] P(Y_d=0\given Y_1=y_1) \\
    & + \sum_{y_1\in \{0,1\} } P(Y_1=y_1\given X) - \sum_{y_1\in \{0,1\} } P(Y_1=y_1) \\
    =& - \sum_{y_1\in \{0,1\} } \bigg[P(Y_1=y_1\given X) - P(Y_1=y_1)\bigg] P(Y_d=0\given Y_1=y_1)
\end{align*}
is the negative of denominator, thus $P(Z\given Y_d=1)=P(Z\given Y_d=0)$ contradicts that $Z\not\indep Y_d$, which completes the proof.
\end{proof}

\section{Proof of Theorem \ref{thm:nopps}}
\label{app:nopps}
First we bound the sample complexity of the backward phase of IAMB, as shown in Algorithm \ref{alg:iamb}.
\begin{lemma}\label{lem:iamb}
Let $\wh{\mkvbdy}_{jk}$ be the output of Algorithm \ref{alg:iamb} with $\ppsthresh\le\entgap_I/2$ and the plug-in mutual information estimator \eqref{eq:mi:bound}. Then
\[
\truepr(\wh{\mkvbdy}_{jk}=\mkvbdy_{jk}=\pa(k))\ge 1 - \frac{|\anc_j|\delta^2_{|\anc_j|}}{\ppsthresh^2}
\]
where $\delta_{\anc_j}^2$ follows the definition in \eqref{eq:def:delta} with $|S|$ replaced by $|\anc_j|$.
\end{lemma}
\begin{proof}
Following the same argument in \eqref{eq:mi:bound}, we have
\[
\sup_{\ell\in\anc_j}\truepr\left(\left|\widehat{I}(X_k;X_\ell\given \anc_j\setminus X_\ell) - I(X_k;X_\ell\given \anc_j\setminus X_\ell)\right| \ge t\right) \le \frac{\delta^2_{|\anc_j|}}{t^2}
\]
Note that for nodes $X_\ell$ not in $\mkvbdy_{jk}$, 
\[
I(X_k;X_\ell\given \anc_j\setminus X_\ell)=I(X_k;X_\ell\given \mkvbdy_{jk}) = 0
\]
Therefore with probability at least $1 - |\anc_j|\delta^2_{|\anc_j|}/t^2$, we have for all $\ell\in \anc_j$
\[
    |\widehat{I}(X_k;X_\ell\given \anc_j\setminus X_\ell) - I(X_k;X_\ell\given \anc_j\setminus X_\ell)| < t
\]
Therefore, for $\ell\in \mkvbdy_{jk}$ and $\ell'\in\anc_j\setminus \mkvbdy_{jk}$,
\[
    \widehat{I}(X_k;X_\ell\given \anc_j\setminus X_\ell) > \entgap_I-t \ \ \ \ \widehat{I}(X_k;X_{\ell'}\given\anc_j\setminus X_{\ell'})<t
\]
Let $t = \ppsthresh \le \entgap_I / 2$, we can remove all $\ell\in \anc_j\setminus \mkvbdy_{jk}$ rather than any one in $\mkvbdy_{jk}$, then the desired Markov boundary is recovered.
\end{proof}
We are now ready to prove Theorem \ref{thm:nopps}.
\begin{proof}[Proof of Theorem \ref{thm:nopps}]
Then we use Lemma~\ref{lem:iamb} and take intersection over all nodes
\[
\truepr\left(\wh{\mkvbdy}_{jk}=\pa(k)\ \ \forall k\in[d]\right)\ge  1 - \frac{d}{\ppsthresh^2}\max_j|\anc_j|\delta^2_{|\anc_j|} \ge 1-\frac{d^2}{\ppsthresh^2}\delta^2_d
\]
where 
\[
\delta_d^2 \asymp \Big(\frac{2^d}{nd}\Big)^2 + \frac{d^2}{n}
\]
The last inequality is due to $|\anc_j|\le |\anc_{\depth-1}| \le d$. Furthermore, use the estimator \eqref{eq:entest} for conditional entropy directly on $\anc_j$ instead of $\mkvbdy_{jk}$,
\[
\sup_{k\in V\setminus\anc_j}\truepr\left(\left|\widehat{\ent}(X_k \given \anc_j) - \ent(X_k\given \anc_j)\right| \ge t\right) \le \frac{\delta_{|\anc_j|}^2}{t^2}.
\]
Following the proof of main theorem in Appendix~\ref{app:main}, we can show that without PPS procedure, we can recover layers with
\[
\truepr(\wh{\layer}=\layer) \ge 1 - \frac{4d\width \depth}{(\entgap/2 \wedge \unentthresh)^2}\max_j\delta_{|\anc_j|}^2 \ge 1 - \frac{4d\width\depth}{(\entgap/2 \wedge \unentthresh)^2}\delta^2_d.
\]
Thus we have recovery for the whole graph with
\begin{align*}
\truepr\left(\wh{\gr}=\gr\right) &= \truepr\left(\wh{\layer}=\layer\right)\truepr\left(\wh{MB}=\pa(k)\ \ \forall k\in[d] \given \wh{\layer}=\layer\right) \\
& \ge 1 - \left(\frac{d^2}{\ppsthresh^2} + \frac{4d\width\depth}{(\entgap/2 \wedge \unentthresh)^2}\right)\delta^2_d\gtrsim 1 - \frac{d\width\depth}{(\min(\ppsthresh, \entgap/2, \unentthresh))^2}\delta^2_d
\end{align*}
Plug in the upper bound of $\delta_d^2$, require for $\truepr(\wh{\gr}=\gr)>1-\epsilon$, we have desired sample complexity.
\end{proof}

\section{Experiment details}\label{app:exp}
We describe the details of experiments conducted in this appendix.

\subsection{Experiment settings}
For graph types, we generate
\begin{itemize}
    \item \textit{Poly-Tree (Tree)}. Uniformly random tree by generating a Pr\"{u}fer sequence with random direction assigned for each edge.
    \item \textit{Erd\H{o}s R\'{e}nyi (ER)}. Graphs whose edges are selected from all possible $\binom{d}{2}$ edges independently with specified expected number of edges. 
    \item \textit{Scale-free networks (SF)}. Graphs simulated according to the Barabasi-Albert model. 
\end{itemize}
For models, we consider the dependency between parents and children. We control the constant conditional entropy $\ent(X_k\given \pa(k))$ for all $k=1,2,\ldots,d$ to satisfy the Condition~\ref{cond:equal}, which implies \ref{cond:uneq:short}. 
We generate data from following models
\begin{itemize}
    \item mod model (MOD): $X_k= (S_k\mod2)^{Z_k}\times (1 - (S_k\mod2))^{1-Z_k}$ where $S_k = \sum_{\ell\in\pa(k)} X_{\ell}$ with $Z_k\sim \BernoulliDist(p)$
    \item additive model (ADD): $X_k = \sum_{\ell\in\pa(k)}X_\ell + Z_k$ with $Z_k\sim\BernoulliDist(p)$
\end{itemize}

Total number of replications is $N=30$. For each of them, we generated random datasets with sample size $n\in\{1000,2000,3000,4000\}$ for graphs with $d \in \{10,20,30,40,50\}$ nodes and $p=0.2$.

\subsection{Implementation and baselines}
We implement our algorithm with entropy estimator proposed in \citet{wu2016minimax}, which is available at \url{https://github.com/Albuso0/entropy}. We treat joint entropy as a multivariate discrete variable to estimate. We fix $\ppsthresh=0.005$ and $\unentthresh=0.001$, in particular, we do no hyperparameter tuning.

We further compare following DAG learning algorithms as baselines:
\begin{itemize}
    \item PC algorithm is standard structure learning approach
    and  The implementation is available at \url{https://github.com/bd2kccd/py-causal}. The independence test is chosen as discrete BIC test \texttt{dis-bic-test}. Remaining parameters are set as default or recommended in tutorial.
    \item Greedy equivalence search (GES) is standard baseline for DAG learning. The implementation is available at \url{https://github.com/bd2kccd/py-causal}. The score is chosen as discrete BIC score \texttt{dis-bic-score}. Remaining parameters are set as default or recommended in tutorial.
\end{itemize}
The experiments were conducted on an internal cluster using an Intel E5-2680v4 2.4GHz CPU with 64 GB memory.

\end{document}